\documentclass[11pt,a4paper]{amsart}
\usepackage{amssymb}
\usepackage{ifpdf}
\usepackage{color}
\usepackage{graphicx,graphics}
\usepackage{amsmath,amscd, amsthm}
\usepackage{tikz}
\usepackage[all,cmtip]{xy}
\usepackage{enumerate}
\usepackage[mathscr]{eucal}
\usepackage{amsbsy}
\usepackage{bm}
\usepackage{float}
\usetikzlibrary{arrows}
\usetikzlibrary{decorations.markings}

\newcommand{\Homeo}{\mathrm{Homeo}}
\newcommand{\OB}{\mathrm{OB}}
\newcommand{\ROB}{\mathrm{ROB}}
\newcommand{\Mod}{\mathrm{Mod}}
\newcommand{\MT}{\mathcal{MT}}
\newcommand{\Z}{\mathbb{Z}}
\renewcommand{\O}{\mathcal{O}}
\renewcommand{\P}{\mathcal{P}}
\DeclareMathOperator{\lcm}{lcm}

\newtheorem{theorem}{Theorem}[section]
\newtheorem{prop}[theorem]{Proposition}
\newtheorem{lemma}[theorem]{Lemma}

\newtheorem{cor*}{Corollary}
\newtheorem{prop*}{Proposition}
\theoremstyle{definition}
\newtheorem{definition}[theorem]{Definition}
\newtheorem{remark}[theorem]{Remark}
\newtheorem{exmp}[theorem]{Example}

\begin{document}
\title{Periodic surface homeomorphisms and contact structures}
	
	\subjclass{Primary 57K33, 57D20; Secondary 57M50}
	\keywords{periodic map, rational open book, contact structure.}
	
	\author{Dheeraj Kulkarni}
	\address{Department of Mathematics\\
Indian Institute of Science Education and Research Bhopal\\
Bhopal Bypass Road, Bhauri \\
Bhopal 462 066, Madhya Pradesh\\
India}
	\email{dheeraj@iiserb.ac.in}
	\urladdr{https://sites.google.com/a/iiserb.ac.in/homepage-dheeraj-kulkarni}
	
	\author{Kashyap Rajeevsarathy}
	\address{Department of Mathematics\\
Indian Institute of Science Education and Research Bhopal\\
Bhopal Bypass Road, Bhauri \\
Bhopal 462 066, Madhya Pradesh\\
India}
	\email{kashyap@iiserb.ac.in}
	\urladdr{https://home.iiserb.ac.in/$\sim$lkashyap}
	
	\author{Kuldeep Saha}
	\address{Department of Mathematics\\
Indian Institute of Science Education and Research Bhopal\\
Bhopal Bypass Road, Bhauri \\
Bhopal 462 066, Madhya Pradesh\\
India}
	\email{kuldeep.saha@gmail.com}
	
\begin{abstract} 
Periodic surface homemorphisms (diffeomorphisms) play a significant role in the the Nielsen-Thurston classification of surface homeomorphisms. Periodic surface homeomorphisms can be 
described (up to conjugacy) by using data sets which are combinatorial objects.  In this article, we start by associating a rational open book to a slight modification of a given data set, called marked data set. It is known that  every rational open book supports a contact structure. Thus, we can associate a contact structure to a periodic map and study the properties of it in terms combinatorial conditions on marked data sets.

 In particular, we prove that a class of data sets, satisfying easy-to-check combinatorial hypothesis, gives rise to Stein fillable contact structures.
In addition to the above, we prove an analogue of Mori's construction of explicit symplectic filling for rational open books.  We also prove a sufficient condition for Stein fillability of rational open books analogous to the positivity of monodromy in honest open books as in the result of Giroux and Loi-Piergallini.
\end{abstract}
	
\maketitle
	
\section{Introduction}
\label{sec:intro}
The set of periodic surface homeomorhphisms is an important class in the scheme of Nielsen-Thurston classfication theory of surface homeomorphisms. The conjugacy class of a periodic surface homeomorphism can be efficiently encoded~\cite{PRS,RV} by a combinatorial tuple of integers called a \emph{data set}. More specifically, consider a homeomorphism $h$ of order $n$ on a closed orientable surface $\Sigma$ that generates an action with $\ell$ distinct nontrivial orbits of sizes $n/n_i$ where for $1 \leq i \ell$, the induced local rotation angle is $2\pi c_i^{-1}/n_i$ with $\gcd(c_i,n_i)=1$, and whose quotient orbifold has genus $g_0$. Then $f$ has an associated data set given by $D_h = (n,g_0,r;(c_1,n_1),\ldots,(c_{\ell},n_{\ell})),$ where $n$ is called the \textit{degree} of $D_h$. The parameter $r$ comes into play only when $f$ is a free rotation of the surface by $2\pi r/n$. For example, the hypereliptic involution on the torus, which has $4$ fixed points (as illustrated in Figure~\ref{Fig 0} below) with an induced local rotation of $\pi$ around each point, is encoded by the data set $(2,0 ;(1,2),(1,2),(1,2),(1,2))$.
	\begin{figure}[H] \label{Fig 0}
		\centering
	\def\svgwidth{\textwidth}
	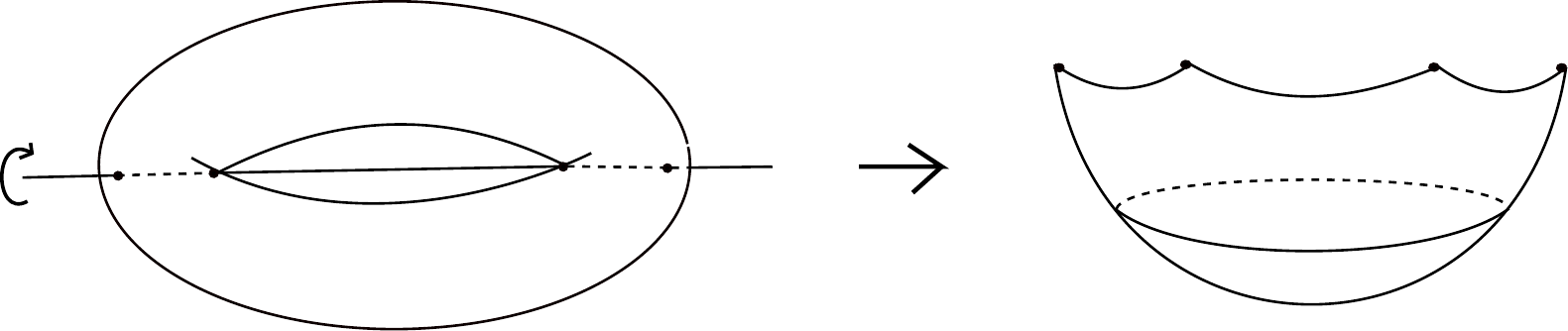
	\caption{The hyperelliptic involution on a torus.}
	\end{figure} 

The open book decomposition of a manifold has proved to be a fundamental tool for development of contact topology. Roughly speaking, an open book is a decomposition of a manifold into a co-dimension $1$ mapping torus (see Section~\ref{sec:prelims} for details) and a tubular neighborhood of a codimension $2$ submanifold, called the \textit{binding}. The fiber of the mapping torus is called the \emph{page}. The return map or \emph{monodromy} of the fibration over $S^1$ is a diffeomorphism of the page relative to its boundary. It is known~\cite{W} that all odd-dimensional manifolds admit open book decompositions.
	
Thurston-Winkelnkemper \cite{TW} showed how to associate a contact structure to an open book with a symplectic page and with a relative symplectomorphism as monodromy. Subsequently, a remarkable result of Giroux \cite{Gi} in dimension three showed that every contact structures arises in this way. Moreover, for $3$-manifolds, Giroux established a correspondence between contact structures up to isotopy and open books up to positive stabilization. Baker-Etnyre-van Horn-Morris \cite{BEM} defined a generalization of $3$-dimensional open books called \emph{rational open books}. While the monodromy of a \emph{honest} open book is identity near the boundary, a rational open book may have rotations along boundary components of the page. It was shown in \cite{BEM} that one can extend the Thurston-Winkelnkemper construction to the setting of rational open books and associate a unique contact structure to a rational open book. 

It is an interesting pursuit to analyze the contact topological properties of rational open books in relation to their monodromies. In this paper, we discuss a natural family of rational open books arising from periodic surface homeomorphisms and study the fillability properties of the corresponding contact structures. 
	
We now describe how to arrive at a contact structure starting from a periodic surface homeomorphism. We go back to our earlier example of the periodic homeomorphism $h$. By removing (cyclically permuted) $h$-invariant open disks around (points in) a set $S$ of nontrivial orbits of a periodic homeomorphism, we obtain a periodic map $\hat{h}$ on a subsurface $(\hat{\Sigma},\partial \hat{\Sigma})$ with corresponding boundary. One can then take the rational open book which has $\hat{\Sigma}$ as page and $\hat{h}$ as monodromy, and consider the contact structure associated to this rational open book, as described in \cite{BEM}. Thus, we see that by removing rotating open disks around the orbits of a cyclic action (encoded by $D_h$), we can get a contact rational open book. In order to make such an association well-defined, the $D_h$ should also include the information on the specific orbits in $S$ around which disks were removed. Note that each pair $(c_j,n_j)$ in the multiset $\{(c_1,n_1),\ldots,(c_{\ell},n_{\ell})\}$ corresponds to a distinguished orbit of the $\langle f\rangle$-action. Thus, we modify $D_h$ to a \textit{marked data set}, which is a tuple of the form
$$\hat{D}_h = (n_{\pm},g_0,r;(c_1,n_1),\ldots,(c_{\ell},n_{\ell}),[j_1,\ldots,j_k] ),$$ where $\{j_1,\ldots, j_k\} \subset \{1,\ldots \ell\}$ and the (disjoint) union of the orbits that the $(c_{j_i},n_{j_i})$ correspond to, equals $S$. For example, assuming that the four fixed points hyperlliptic involution (Figure \ref{Fig 0}) are marked $1$-$4$,  and say we remove three disks around points $1, 2$ and $4$, then the marked data set is given by: $(n,0; (1,2), (1,2), [1,2,4])$. The suffix $\pm$ for the parameter $n$ is added to distinguish between a positive (i.e clockwise) and a negative (i.e counterclockwise) local rotation around a distinguished fixed point (or orbit). We show the following result in Section \ref{marked data sets}. 

\begin{prop}\label{association thm}
Given a marked data set $\hat{D}$, one can associate a unique contact structure to it.
\end{prop}
\noindent The above proposition opens the door to the possibility of describing a class of contact structures combinatorially in terms of marked data sets. Moreover, the properties of contact structures can be studied through marked data sets. Guided by this philosophy, our main results are formulated in terms of conditions on marked data sets. 

In \cite{PRS}, Prasad-Rajeevsarathy-Sanki described a procedure for decomposing an arbitrary periodic map $h$ of order $n$ on $\Sigma$ (that is not realizable as a rotation of $\Sigma$) into irreducible (periodic) components $h_i$ (of order $n$) on surfaces $\Sigma_i$ that have that have at least one fixed point, called \textit{irreducible Type 1 maps}. It is known~\cite{G1} that an irreducible Type 1 map $f_0$ is encoded by a data set of the form $D_{f_0}=(n,0,(c_1,n_1),(c_2,n_2),(c_3,n)$. Further, it was shown that a map such as $f_0$ is realized as a rotation of a special hyperbolic polygon (with side-pairing) by $2\pi c_3^{-1}/n$. The decomposition of $h$ into irreducibles induces a decomposition of $D_h$ into simpler data sets $D_{h_i}$ of the form $D_{f_0}$. This, in turn, extends to a decomposition of $\hat{D}_h$ into marked irreducible Type 1 data sets $\hat{D}_{h_i}$. This process (of decomposing $h$ into irreducibles) can be completely reversed (to recover the original action $h$) by \textit{gluing} pairs of such irreducible components along \textit{compatible orbits}, where the induced local rotation angles are equal. More precisely, when the actions of a pair $(h_1,h_2)$ of irreducible Type 1 components induce \textit{compatible orbits of size k} (i.e have orbits where the local rotation angles add up to $0 $modulo $2 \pi$), we can remove (cyclically permuted) invariant disks around points in the orbits and then identify the resultant boundary components, thereby realizing a new action $(h_1,h_2)$. We call this process of constructing $(h_1,h_2)$ as a \textit{$k$-compatibility}. For example, the order $6$ action $f$ on the torus with data set $D_{f} = (6, 0; (1,2),(1,3),(1,6))$ is compatible with the order $6$-action $f^5$ on the torus (represented by $D_{f^5} = (6, 0; (1,2),(2,3),(5,6))$) along the fixed points at the center of the realizing hexagons, resulting in the compatible pair $(f,f^5)$ represented by $(D_f,D_{f^5}) = (6; 0, (1,2), (1,2), (1,3), (2,3))$. A similar compatibility of gluing can be also defined for marked data sets.  

In Section~\ref{sec:symp_fill}, we investigate the fillability properties of the contact structures associated to marked data sets. The following result describes a class of Stein fillable contact structures in terms of their marked data sets.   

\begin{theorem}\label{stein filling 1}
Let $\hat{D}$ be a marked data set representing an action of degree $n$ on $\Sigma$ that decomposes into a collection of marked irreducible Type 1 data sets 
$$\hat{D}_i = (n_+,0,(c_{i1},n_{i1}),(c_{i2},n_{i2}),(c_{i3},n),[j_{i1},\ldots,j_{ik}])$$
representing actions on surfaces $\Sigma_{g_i}$, for $1 \leq i \leq m$,  such that the following conditions hold: 
\begin{enumerate}[(i)]
\item If $s \in \{j_{i1},\ldots,j_{ik}\}$, then $n_s =n$.
\item If $(D_i,D_j)$ forms a compatible pair, then both $\partial (\Sigma_{g_i} \cup \Sigma_{g_j}) \cap \partial \Sigma_{g_i}$  and $\partial (\Sigma_{g_i} \cup \Sigma_{g_j}) \cap \partial \Sigma_{g_j}$ are non-empty.
\end{enumerate}	
Then the contact structure associated to $\hat{D}$ is Stein fillable. 			
\end{theorem}	

\noindent For example, the marked data set $\hat{D}_f = (6,0;(1,2),(1,3),(1,3),(5,6),[4])$ that is realized as a compatible pair $(D_f,D_{f^5})$ of irreducible marked data sets satisfies the hypotheses of Theorem~\ref{stein filling 1}. Therefore, the contact structure associated to $\hat{D}_f$ is Stein fillable. 
\begin{remark}
Theorem \ref{stein filling 1} can be viewed as a generalization of a result due to Colin and Honda \cite[Theorem 4.2]{CH}. The result of Colin-Honda for honest open book says that a contact structure supported by an open book with a periodic monodromy (i.e. having a periodic map as its Thurston representative) is Stein fillable if the monodromy is right-veering. 
\end{remark}
Our next result pertains to producing explicit symplectic fillings of some rational open books. In \cite{Mo}, Mori constructed such symplectic fillings for the contact structures coming from an explicit construction of Thurston and Winkelnkemper. We give a generalization of this construction for rational open books. For the notations regarding rational open books and its monodromy we refer to section \ref{integral to rational}. A rational open book is characterized by a surface $\Sigma$ with boundary and a homeomorphism $\phi$ of $\Sigma$. We denote such a rational open book by $\ROB(\Sigma,h)$. Let $\Mod(\Sigma)$ denote the mapping class group of $\Sigma$. If $\phi$ does not permute boundaries, then $\phi$ can be written as a composition of some element $h \in \Mod(\Sigma)$ and fractional rotations of the boundary components~\cite[Section 5]{BEM}. Let $\mathrm{Dehn}^+(\Sigma,\partial \Sigma)$ consist of relative isotopy classes of diffeomorphisms which are products of positive Dehn twists. The following result is stated for surfaces with connected boundary, but the proof holds true for surfaces with multiple boundary components.  

\begin{theorem} \label{symp filling 2}
Consider a rational open book $\ROB(\Sigma,\phi)$ with $\partial \Sigma$ connected. Say, $\phi$ can be written as $h \circ \partial_{\frac{q}{p}}$, where $h \in \mathrm{Dehn}^+(\Sigma,\partial \Sigma)$ and $\partial_{\frac{q}{p}}$ is a $\frac{2\pi q}{p}$-rotation along $\partial \Sigma$. Then $\ROB(\Sigma,\phi)$ admits a strong symplectic filling for $p > q > 0$.
\end{theorem}

On similar lines, a result by Giroux \cite{Gi}  and Loi-Piergallini \cite{Loi_Pier} for Stein-fillability of open book can be generalized to the setting of rational open books as stated follows.

\begin{theorem}\label{Stein filling 2}
Let $\ROB(\Sigma,\phi)$ be a rational open book  with $\Sigma$ connected. Suppose $\phi$ can be written as $h \circ \partial_{\frac{q_i}{p_i}}$, where $h \in \mathrm{Dehn}^+(\Sigma,\partial \Sigma)$ and $\partial_{\frac{q_i}{p_i}}$ is a $\frac{2\pi q_i}{p_i}$-rotation along $i$-th boundary component in $\partial \Sigma$. Then $\ROB(\Sigma,\phi)$ admits a Stein filling if for all $i$, $p_i > q_i> 0$.
\end{theorem} 

With the help of above results, one can produce many examples of fillable contact structures explicitly. The above mentioned results begin the exploration of similar results concerning fillability for the case of pseudo-periodic surface homeomorphisms,  which the authors are currently pursuing.

\bigskip
\noindent {\bf Acknowledgements:} The work in this article is supported by the grant EMR/2017/000727 by SERB, Government of India. The first author would also like to thank James Conway for noting an error in the earlier draft of the article and helpful conversations.
 
\section{Preliminaries}	
\label{sec:prelims}
	
\subsection{Periodic maps on surfaces} 
\label{subsec:periodic map}
Let $\Sigma$ be an oriented compact surface, and let $\Mod(\Sigma)$ denote the mapping class group of $\Sigma$. We recall the Neilsen-Thurston classification~\cite{Th} of surface diffeomorphisms.
\begin{theorem}
 An $f \in \Mod(\Sigma)$ is represented by a homeomorphism $h$ such that at least one of the following holds.
\begin{enumerate}
\item $h$ is a periodic.
\item $h$ is a reducible (i.e. $h$ preserves a multicurve in $\Sigma$).
\item $h$ is a pseudo-Anosov 
\end{enumerate}
\end{theorem}

\noindent Suppose that $\Sigma_g$ is a closed oriented surface with genus $g \geq 1$. By the Nielsen realization theorem~\cite{Ke}, an $f \in \Mod(\Sigma)$ of order $n$ is represented by a homeomorphism $h$ of the same order.  Note that we will refer to both $h$ and $\langle h \rangle$, interchangeably, as a \textit{$\Z_n$-action on $\Sigma_g$}. Suppose that the quotient orbifold $\O_h := \Sigma_g/\langle h \rangle$ has $\ell$ cone points $x_i$, $1 \leq i \leq \ell$. Then each $x_i$ lifts to an orbit of size $n/n_i$ on $\Sigma$ (of the $\langle h \rangle$-action), and the local rotation induced by $h$ around the points in each orbit is given by $2 \pi c_i^{-1}/n_i$, where $c_i c_i^{-1} \equiv 1 \pmod{n_i}$. We will now formalize the notion of data set introduced in Section~\ref{sec:intro}.  

\begin{definition}
\label{defn:data_set}
	A \textit{data set of degree $n$} is a tuple
	$$
	D = (n,g_0, r; (c_1,n_1), (c_2,n_2),\ldots, (c_{\ell},n_{\ell})),
	$$
	where $n\geq 1$, $ g_0 \geq 0$, and $0 \leq r \leq n-1$ are integers, and each $c_i$ is a residue class modulo $n_i$ such that:
	\begin{enumerate}[(i)]
		\item $r = 0$ if, and only if $\ell = 0$, and when $r >0$, we have $\gcd(r,n) = 1$, 
		\item each $n_i\mid n$,
		\item for each $i$, $\gcd(c_i,n_i) = 1$, 
		\item for each $i$, $\lcm(n_1,\ldots \widehat{n_i}, \ldots,n_{\ell}) = \lcm(n_1,\ldots,n_{\ell})$, and $\lcm(n_1,\ldots,n_{\ell}) = n$, if $g_0 = 0$,  and
		\item $\displaystyle \sum_{j=1}^{\ell} \frac{n}{n_j}c_j \equiv 0\pmod{n}$.
	\end{enumerate}
	The number $g$ determined by the equation
	\begin{equation*}\label{eqn:riemann_hurwitz}
	\frac{2-2g}{n} = 2-2g_0 + \sum_{j=1}^{\ell} \left(\frac{1}{n_j} - 1 \right) \tag{R-H}
	\end{equation*}
	is called the \emph{genus} of the data set.
\end{definition}

\noindent The following proposition (see \cite[Theorem 3.9]{RV}) allows us to represent the conjugacy of a cyclic action by a data set. 

\begin{prop}\label{prop:ds-action}
Data sets of degree $n$ and genus $g$ correspond to conjugacy classes of $C_n$-actions on $\Sigma_g$. 
\end{prop}

\noindent We will denote the data set associated with a periodic map $h$ by $D_h$. For a $D_h$ as in Definition~\ref{defn:data_set}, the integer $r$ will be non-zero only when $h$ is a free rotation of $\Sigma_g$ by $2\pi r/n$, in which case $D_h$ would take the form $(n,g_0,r;-)$. Equation~\ref{eqn:riemann_hurwitz} in Definition~\ref{defn:data_set} is the Riemann-Hurwitz equation associated with the branched covering $\Sigma_g \to \mathcal{O}_h$. Before diving into the geometric realizations of  cyclic actions, we recall from \cite{PRS} the classification of $C_n$-actions on $\Sigma_g$ into three broad categories. 

\begin{definition}\label{def:types_of_actions}
Let $h$ be a $C_n$-action on $\Sigma_g$ with $D_h$ as in Definition~\ref{defn:data_set}. Then $h$ is said to be a: 
\begin{enumerate}[(i)]
\item \textit{rotational action}, if either $r \neq 0$, or $D$ is of the form 
$$(n,g_0;\underbrace{(s,n),(n-s,n),\ldots,(s,n),(n-s,n)}_{k \,pairs}),$$ 
for integers $k \geq 1$ and $0<s\leq n-1$ with $\gcd(s,n)= 1$, and $k=1$ if, and only if $n>2$.  
\item \textit{Type 1 action}, if $\ell = 3$ and $n_3 = n$. 
\item \textit{Type 2 action}, if $h$ is neither a rotational nor a Type 1 action. 
\end{enumerate}
\end{definition}  

\noindent It is apparent that rotational actions can be realized as the rotations of $\Sigma_g$ through an axis under a suitable isometric embedding $\Sigma_g \to \mathbb{R}^3$. Moreover, Gilman~\cite{G1} showed that a Type 1 $C_n$-action $h$ on $\Sigma_g$ is \emph{irreducible} if and only if $\mathcal{O}_h$ is a sphere with three cone points, which in the language of data sets means that 
$$D_h = (n,0;(c_1,n_1),(c_2,n_2),(c_3,n)).$$

\noindent In~\cite{PRS}, it was shown that every irreducible Type 1 action is as a rotation of a unique polygon with side-pairing.   

\begin{theorem}\label{res:1}
For $g \geq 2$, consider a irreducible Type 1 action $f$ on $\Sigma_g$ with $$D_f = (n,0; (c_1,n_1),\linebreak (c_2,n_2), (c_3,n)).$$ Then $f$ can be realized explicitly as the rotation $\theta_f$ of a unique hyperbolic polygon $\P_f$ with a suitable side-pairing $W(\P_f)$, where $\P_f$ is a hyperbolic  $k(f)$-gon with
$$ k(f) := \begin{cases}
2n, & \text { if } n_1,n_2 \neq 2, \text{ and } \\
n, & \text{otherwise, }
\end{cases}$$
and for $0 \leq m\leq n-1$, 
$$ 
W(\P_f) =
\begin{cases}
\displaystyle  
  \prod_{i=1}^{n} a_{2i-1} a_{2i} \text{ with } a_{2m+1}^{-1}\sim a_{2z}, & \text{if } k(h) = 2n, \text{ and } \\
\displaystyle
 \prod_{i=1}^{n} a_{i} \text{ with } a_{m+1}^{-1}\sim a_{z}, & \text{otherwise,}
\end{cases}$$
where $\displaystyle z \equiv m+qj \pmod{n}$ and $q= (n/n_2)c_3^{-1},j=n_{2}-c_{2}$. 
\end{theorem}


\noindent Further, it was shown that an arbitrary non-rotational cyclic action $h$ admits a decomposition into irreducible Type 1 components. Conversely, given such a decomposition $h$ can be recovered by piecing together the irreducible components using the following two types of constructions.
\begin{enumerate}[(a)]
\item \textit{$k$-compatibility.} This construction involves the removal of cyclically permuted invariant disks around points in a pair of \textit{compatible orbits}  (where the induced local rotation angles are the same) by a pair $h_i$ of irreducible Type 1 action on $\Sigma_{g_i}$ then identifying the resulting boundary components thereby obtaining an action $h=(h_1,h_2)$ on $\Sigma_{g_1+g_2+k-1}$. This process of constructing $h$ is called a \emph{k-compatibility}. If this construction is performed with a pair of compatible orbits induced by a single action $h$ on $\Sigma_g$, resulting in an action on $\Sigma_{g+k}$, then the process is called a \textit{self $k$-compatibility. }
\item \textit{Permutation additions and deletions.} In the process of \textit{permutation addition} we remove (cyclically permuted) invariant disks around points in an orbit of an action size $n$ induced by the action $h$ and then paste $n$ copies of $\Sigma_{g'}^1$ (i.e. $\Sigma_{g'}$ with one boundary component) to the resultant boundary components. This results in an action on $\Sigma_{g+ng'}$ with the same fixed point and orbit data as $h$. The reverse of this process, which involves the removal of such an added permutation component, is called a \textit{permutation deletion. }
 \end{enumerate}
 
 \noindent The upshot of the preceding discussion is the following. 
 
 \begin{theorem}\label{res:2}
For $g \geq 2$, an arbitrary non-rotational action on $\Sigma_g$ can be constructed through finitely many $k$-compatibilities, permutation additions, and permutation deletions on irreducible Type 1 actions. 
\end{theorem}

 \noindent We will now describe the conjugacy classes of a $k$-compatibe pair $(f_1,f_2)$ of actions in terms of $D_{f_1}$ and $D_{f_2}$.
\begin{definition}
A pair of data sets 
$$D = (n, g_0; (c_1,n_1),\ldots,(c_l,n_l)) \text{ and } \tilde{D} = (\tilde{n}, \tilde{g_0}; (\tilde{c}_1,\tilde{n}_1),\ldots,(\tilde{c}_l,\tilde{n}_{\tilde{l}}))$$  are said to be \emph{$(i,j)$-compatible} if there exist $i,j$ such that
\begin{enumerate}[(i)]
\item $n_i = \tilde{n}_j = m$.
\item $c_i+ \tilde{c}_j \equiv 0 \pmod m$. 
\end{enumerate} 	
Given a pair of $(i,j)$-compatible data sets $D$ and $\tilde{D}$ as above, we define
\begin{gather*}
(D,\tilde{D}) := (m,g_0+\tilde{g}_0; (c_1,n_1),\ldots, \widehat{(c_i,n_i)}, \ldots, (c_l,n_l),  \\ (\tilde{c}_1,\tilde{n}_1), \ldots, \widehat{(\tilde{c}_{i},\tilde{n}_{i})},\ldots,(\tilde{c}_l,\tilde{n}_{\tilde{l}}))
\end{gather*}
\end{definition}

\begin{exmp}\label{(1,5)-compatible exmp}
The $1$-compatible pair $(f,f^5)$ from Section~\ref{sec:intro} is represented by a $(3,3)$-compatible $(D_f, D_{f^5})$ of data sets, where $D_f=(6, 0; (1,2), (1,3), (1,6))$ and $D_{f^5} = (6, 0; (1,2), (2,3), (5,6))$. This construction is illustrated in Figure~\ref{compfig1} below.

\end{exmp}

\begin{figure}
	\centering
\def\svgwidth{\textwidth}
\begingroup%
  \makeatletter%
  \providecommand\color[2][]{%
    \errmessage{(Inkscape) Color is used for the text in Inkscape, but the package 'color.sty' is not loaded}%
    \renewcommand\color[2][]{}%
  }%
  \providecommand\transparent[1]{%
    \errmessage{(Inkscape) Transparency is used (non-zero) for the text in Inkscape, but the package 'transparent.sty' is not loaded}%
    \renewcommand\transparent[1]{}%
  }%
  \providecommand\rotatebox[2]{#2}%
  \newcommand*\fsize{\dimexpr\f@size pt\relax}%
  \newcommand*\lineheight[1]{\fontsize{\fsize}{#1\fsize}\selectfont}%
  \ifx\svgwidth\undefined%
    \setlength{\unitlength}{513.95171449bp}%
    \ifx\svgscale\undefined%
      \relax%
    \else%
      \setlength{\unitlength}{\unitlength * \real{\svgscale}}%
    \fi%
  \else%
    \setlength{\unitlength}{\svgwidth}%
  \fi%
  \global\let\svgwidth\undefined%
  \global\let\svgscale\undefined%
  \makeatother%
  \begin{picture}(1,0.21875801)%
    \lineheight{1}%
    \setlength\tabcolsep{0pt}%
    \put(0,0){\includegraphics[width=\unitlength,page=1]{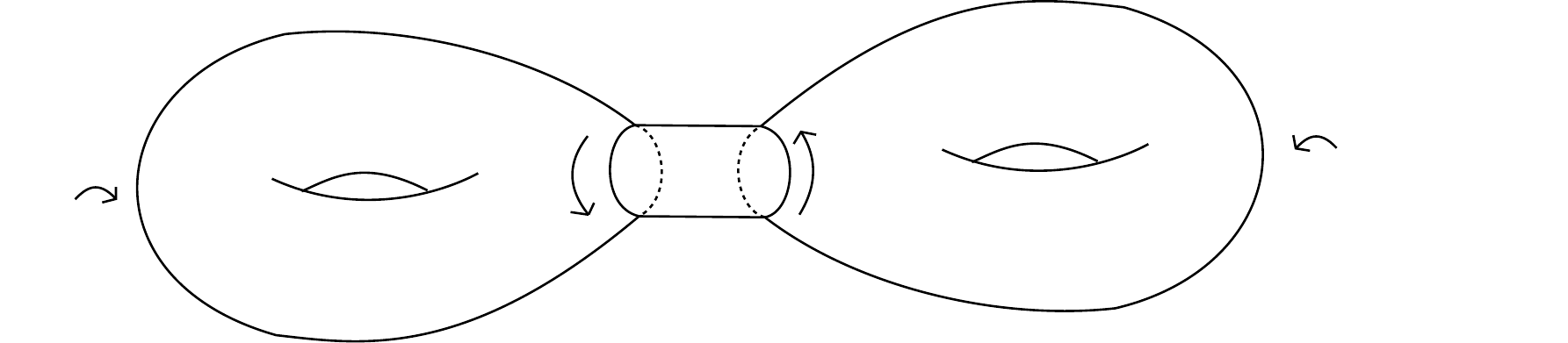}}%
    \put(-0.0013605,0.09008821){\color[rgb]{0,0,0}\makebox(0,0)[lt]{\lineheight{1.25}\smash{\begin{tabular}[t]{l}$\textrm{D}_1$\end{tabular}}}}%
    \put(0.86574058,0.11454131){\color[rgb]{0,0,0}\makebox(0,0)[lt]{\lineheight{1.25}\smash{\begin{tabular}[t]{l}$\textrm{D}_2$\end{tabular}}}}%
    \put(0.32160719,0.10302529){\color[rgb]{0,0,0}\makebox(0,0)[lt]{\lineheight{1.25}\smash{\begin{tabular}[t]{l}$\frac{2 \pi}{6}$\end{tabular}}}}%
    \put(0.52219548,0.09878821){\color[rgb]{0,0,0}\makebox(0,0)[lt]{\lineheight{1.25}\smash{\begin{tabular}[t]{l}$\frac{10 \pi}{6}$\end{tabular}}}}%
  \end{picture}%
\endgroup%

	\caption{A $1$-compatibility between a pair of irreducible order-$6$ maps on the torus.}
	\label{compfig1}
\end{figure}
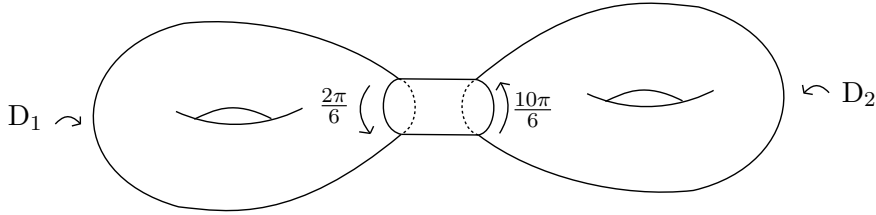

\begin{definition}\label{def:self_comp_ds}
	For $\ell \geq 4$, let
	$$
	D = (n,g_0; (c_1,n_1), (c_2,n_2),\ldots, (c_{\ell},n_{\ell})),
	$$
	be a $C_n$-action. Then $D$ is said to be $(r,s)$-\textit{self compatible}, if there exist $1 \leq r < s \leq \ell$ such that
	\begin{enumerate}[(i)]
		\item $n_r = n_s = m$, and
		\item $\displaystyle c_r+c_s \equiv 0 \pmod{m}$. 
	\end{enumerate}
\end{definition}

\begin{exmp}
	Consider a Type $2$ action given by the data set $D = (3,1; (1,3), (2,3))$. Consider the irreducible Type $1$ data sets: $D_1 = (3,0; (1,3),$\\$ (1,3), (1,3))$ and $D_2 = (3,0; (2,3), (2,3), (2,3))$. $D_1$ can be realized as a $\frac{2\pi}{3}$-rotation on a hexagon with opposite sides identified (i.e. a torus). $D_2$ can be realized as a $\frac{4\pi}{3}$-rotation on a similar hexagon. These two tori can be glued by removing a $(1,2)$-compatible pair of disks and will give a genus $2$ surface with the following data set: $(3,0; (1,3), (1,3), (2,3), (2,3))$. Now we can remove another $(2,3)$-self compatible pair of disk on that genus $2$ surface and attach a tube. This will give us a genus $3$ surface with data set $D = (3,1; (1,3), (2,3))$, as shown in Figure \ref{selfcomp}.
\end{exmp}

\begin{figure} 
	\centering
\def\svgwidth{7.5cm}
\begingroup%
  \makeatletter%
  \providecommand\color[2][]{%
    \errmessage{(Inkscape) Color is used for the text in Inkscape, but the package 'color.sty' is not loaded}%
    \renewcommand\color[2][]{}%
  }%
  \providecommand\transparent[1]{%
    \errmessage{(Inkscape) Transparency is used (non-zero) for the text in Inkscape, but the package 'transparent.sty' is not loaded}%
    \renewcommand\transparent[1]{}%
  }%
  \providecommand\rotatebox[2]{#2}%
  \newcommand*\fsize{\dimexpr\f@size pt\relax}%
  \newcommand*\lineheight[1]{\fontsize{\fsize}{#1\fsize}\selectfont}%
  \ifx\svgwidth\undefined%
    \setlength{\unitlength}{345.62483132bp}%
    \ifx\svgscale\undefined%
      \relax%
    \else%
      \setlength{\unitlength}{\unitlength * \real{\svgscale}}%
    \fi%
  \else%
    \setlength{\unitlength}{\svgwidth}%
  \fi%
  \global\let\svgwidth\undefined%
  \global\let\svgscale\undefined%
  \makeatother%
  \begin{picture}(1,0.43187916)%
    \lineheight{1}%
    \setlength\tabcolsep{0pt}%
    \put(0,0){\includegraphics[width=\unitlength,page=1]{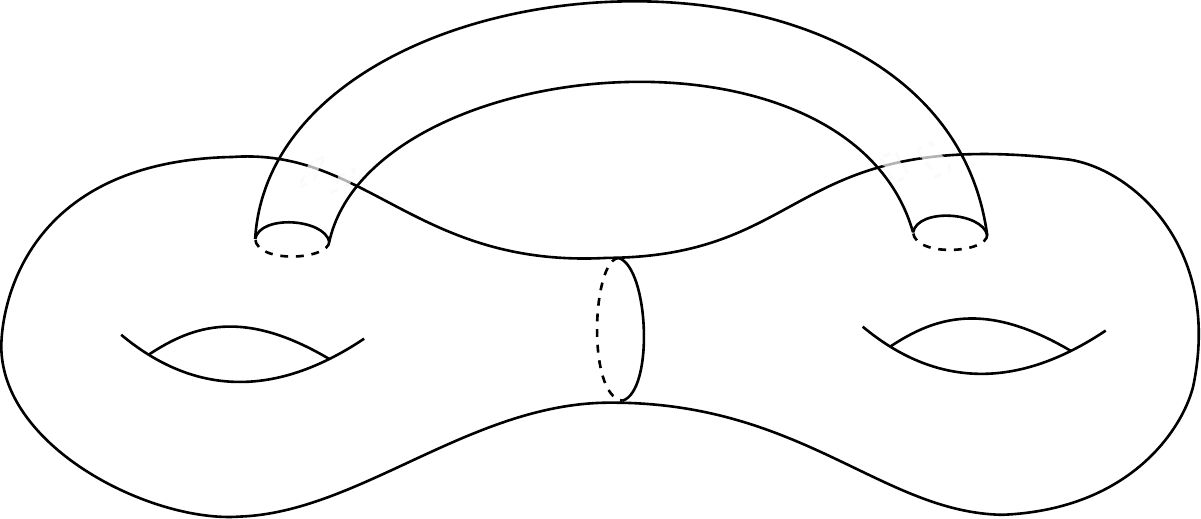}}%
    \put(0.19538645,0.17526836){\color[rgb]{0,0,0}\makebox(0,0)[lt]{\lineheight{1.25}\smash{\begin{tabular}[t]{l}$(1,3)$\end{tabular}}}}%
    \put(0.75618488,0.18058816){\color[rgb]{0,0,0}\makebox(0,0)[lt]{\lineheight{1.25}\smash{\begin{tabular}[t]{l}$(2,3)$\end{tabular}}}}%
  \end{picture}%
\endgroup%

	\caption{Obtaining $(3,1;(1,3),(2,3))$ from $(3,0;(1,3),(1,3),(2,3),(2,3))$ by a $(2,3)$-self compatible gluing.}
	\label{selfcomp}
\end{figure}

\subsection{Fractional Dehn twist coefficient and right veering homeomorphisms} \label{FDTC}
We now only consider elements in $\Mod(\Sigma, \partial \Sigma)$ which are freely isotopic to a periodic or pseudo-periodic homeomorphism. Let $\textrm{Homeo}^+(\Sigma)$ denote the set of orientation preserving homeomorphisms of $\Sigma$.

\begin{definition}
	A map $\phi \in \textrm{Homeo}^+(\Sigma)$, freely isotopic to $h \in \Mod(\Sigma,\partial \Sigma)$, is called a \textit{Thurston representative of} $h$, if it is periodic.
\end{definition} 

\noindent Unlike $h$, $\phi$ may rotate the boundary components of $\Sigma$. If $C \subset \partial \Sigma$ is a boundary component, $\phi\vert_{C}$ is given by a $\frac{2 \pi q}{p}$-rotation for some $p \in \mathbb{Z}_{>0}$ and $q \in \mathbb{Z}$ such that $\vert q \vert \leq p$. The rational number $c(h) = \frac{q}{p}$ is called the \emph{fractional Dehn twist coefficient} ($\textrm{FDTC}$) of $h$ with respect to $C$. For example, the hyperelliptic involution on $\Sigma_g^1$ has FDTC equal to $\frac{1}{2}$.
 
Consider the mapping torus determined by $\phi$. The induced flow, when restricted to a boundary component $C$ of the page has periodic orbits. Let $\gamma$ be one such orbit. Then one can write $\gamma$, in terms of the meridian $\mu$ and longitude $\lambda = C$, as $\gamma = p\lambda + q\mu$, where $p$, $q$ are relatively prime integers.

\begin{definition}\label{fracdehn}
	
	The \emph{fractional Dehn twist coefficient} (FDTC) of $h$ with respect $\partial \Sigma$ is given by $c(h) = \frac{q}{p}$.
	
\end{definition} 

 Note that here we follow the slope convention in~\cite{BEM} which interchanges $p$ and $q$ in the definition of $c(h)$ in \cite{HKM}. The definition is analogous in the case of multiple boundary components, one can similarly define a \emph{fractional Dehn twist coefficient $c_i$} for
the $i$th boundary component.

Following the notation in~\cite{CH}, let $H : \Sigma \times [0, 1] \rightarrow \Sigma$ be the free isotopy from $h(x) = H(x, 0)$ to its periodic representative $\phi(x) = H(x, 1)$. Define $\beta : \partial \Sigma \times [0, 1] \rightarrow \partial \Sigma \times [0, 1]$ by sending $(x, t) \rightarrow (H(x,t), t).$ We form the union of $\partial \Sigma \times [0,1]$ and $\Sigma$ by gluing $\partial \Sigma \times \{1\}$ and $\partial \Sigma$. We then identify this union with $\Sigma$ to construct the homeomorphism $\bar{\beta} \cup \phi$ on $\Sigma$ which is isotopic to $h$ relative to $\partial \Sigma$. Here, $\bar{\beta}$ denotes the reverse isotopy from $t = 1$ to $t = 0$. We will assume that $h = \bar{\beta} \cup \phi$. In terms of this description of $h$, the FDTC is given by the rotation induced by $\beta(x_0,1) - \beta(x_0,0)$. 

We will briefly recall the notion of \emph{right-veering} homeomorphisms from~\cite{CH}. Let $\alpha$ and $\beta$ be isotopy classes, relative to end points, of properly embedded oriented arcs $[0, 1] \rightarrow \Sigma$ with a common initial point $\alpha(0) = \beta(0) = x \in \partial \Sigma$. Choose representatives $a$, $b$ of $\alpha$, $\beta$ ($\alpha \neq \beta$), respectively, so that they intersect transversely (including endpoints) and with minimal number of intersections. Then we say $\beta$ is strictly to the \emph{right} of $\alpha$ if the tangent vectors $(b'(0), a'(0))$ define the orientation
on $\Sigma$ at $x$. A monodromy map $h$ is \emph{right-veering} if for every choice of base point $x \in \partial \Sigma$ and every
choice of arc $\alpha$ based at $x$, either $h(\alpha) = \alpha$ or $h(\alpha)$ is strictly to the right of $\alpha$.
The following proposition (from \cite{HKM}) relates right-veering maps to the FDTC.
\begin{prop}
Let $h \in \Mod(\Sigma, \partial \Sigma)$. Then $h$ is right-veering if and only if $c(h) \geq 0$ for every component of $\partial \Sigma$. Similarly, $h$ is left-veering if and only if $c(h) < 0$ for every component of $\partial \Sigma$.
\end{prop} 
	
\subsection{Rational open book decomposition}
	
	A \emph{rational open book decomposition} for a manifold $M$ is a pair $(L,\pi)$ consisting of an oriented link $L \subset M$ and a fibration $\pi : M \setminus L \rightarrow S^1$ such that, if $N$ is a small tubular neighborhood of $L$, then no component of $\partial N \cap \pi^{-1}(\theta)$ is a meridian of a component of $L$. As mentioned in \cite{BEM}, a rational open book may differ from an \emph{honest} open book in the following two ways.
	\begin{enumerate}[(a)]
		\item A component of $\partial N \cap \pi^{-1}(\theta)$ does not have to be a longitude to a component of $L$.
		\item A component of $\partial N$ intersected with $\pi^{-1}(\theta)$ does not have to be connected.
	\end{enumerate}	

   As in the case of an honest open book, $\Sigma = \overline{\pi^{-1}(\theta)}$ is called a \emph{page} of the rational open book for any $\theta \in S^1$ and $L$ is called the \emph{binding}. Similar to the honest open book, one can describe a rational open book using the monodromy map $\phi_0 : \Sigma \rightarrow \Sigma $ of the fibration $\pi$. In case of honest open books the monodromy map is assumed to be identity near the boundary of a page. For rational open books we require that near boundary $\phi_0^m = id$ for some integer $m$.
   
\subsection{Contact structures on rational open books}\label{contact structure on rob}
We say a rational open book $(L, \pi)$ for $M$ supports a contact structure $\xi$ if there is a
contact form $\alpha$ for $\xi$ such that the following conditions hold.
\begin{enumerate}
	\item $\alpha(v) > 0$ for all positively pointing tangent vectors $v \in TL$.
	\item $d\alpha$ is a volume form when restricted to each page of the open book.
\end{enumerate}
	
\noindent We recall the Thurston-Winkelnkemper construction of contact structure on a rational open book from \cite{BEM}.	
	
Let $(\Sigma, \lambda)$ be the page of a rational open book, where $\lambda$ is a one form such that $d\lambda$ is an area form on $\Sigma$ and $\lambda = r d\theta$ near $\partial \Sigma \subset (\partial \Sigma \times [-1, -1 + \epsilon], (\theta, r))$ for some sufficiently small $\epsilon > 0$. Let $\lambda_{(t,x)} = t \lambda_x + (1-t) (\phi^*\lambda)_x$ be a $1$-form on $(\Sigma \times [0,1], (x,t))$. Then the $1$-form $\alpha_K = \lambda_{(t,x)} + K dt$ is contact for $K$ large enough. Thus, $\alpha_K$ defines a contact form on the mapping torus $\MT(\Sigma,\phi)$. Next, we extend this $1$-form over the solid tori neighborhood of the binding. Here we describe the extension for a single binding component. Let $(S^1 \times D^2, (\theta, r, \phi))$ be a neighborhood of a binding component. Consider the following gluing map between a boundary neighborhood $S^1 \times N(\partial D^2)$ of $S^1 \times D^2$ and a boundary neighborhood $N(\partial \MT)$ of the mapping torus.

\begin{alignat*}{2}
\psi: S^1 \times N(\partial D^2) &\longrightarrow&  N(\partial \MT)\\
(\theta,r,\phi) &\longmapsto& \   \ (-r, p\theta + q\phi, -q\theta + p \phi) 
\end{alignat*}

\noindent An example of this gluing map, which is defined for $r \in [1-\epsilon, 1]$, is illustrated in Figure~\ref{fig:gluing_map} below. Figure \ref{permglue} represents the gluing region for an orbit of $3$ cone points constituting $3$ permuting boundary components. 

\begin{figure}[htbp]
    \centering
\def\svgwidth{\textwidth}
\begingroup%
  \makeatletter%
  \providecommand\color[2][]{%
    \errmessage{(Inkscape) Color is used for the text in Inkscape, but the package 'color.sty' is not loaded}%
    \renewcommand\color[2][]{}%
  }%
  \providecommand\transparent[1]{%
    \errmessage{(Inkscape) Transparency is used (non-zero) for the text in Inkscape, but the package 'transparent.sty' is not loaded}%
    \renewcommand\transparent[1]{}%
  }%
  \providecommand\rotatebox[2]{#2}%
  \newcommand*\fsize{\dimexpr\f@size pt\relax}%
  \newcommand*\lineheight[1]{\fontsize{\fsize}{#1\fsize}\selectfont}%
  \ifx\svgwidth\undefined%
    \setlength{\unitlength}{426.42065891bp}%
    \ifx\svgscale\undefined%
      \relax%
    \else%
      \setlength{\unitlength}{\unitlength * \real{\svgscale}}%
    \fi%
  \else%
    \setlength{\unitlength}{\svgwidth}%
  \fi%
  \global\let\svgwidth\undefined%
  \global\let\svgscale\undefined%
  \makeatother%
  \begin{picture}(1,0.27454699)%
    \lineheight{1}%
    \setlength\tabcolsep{0pt}%
    \put(0,0){\includegraphics[width=\unitlength,page=1]{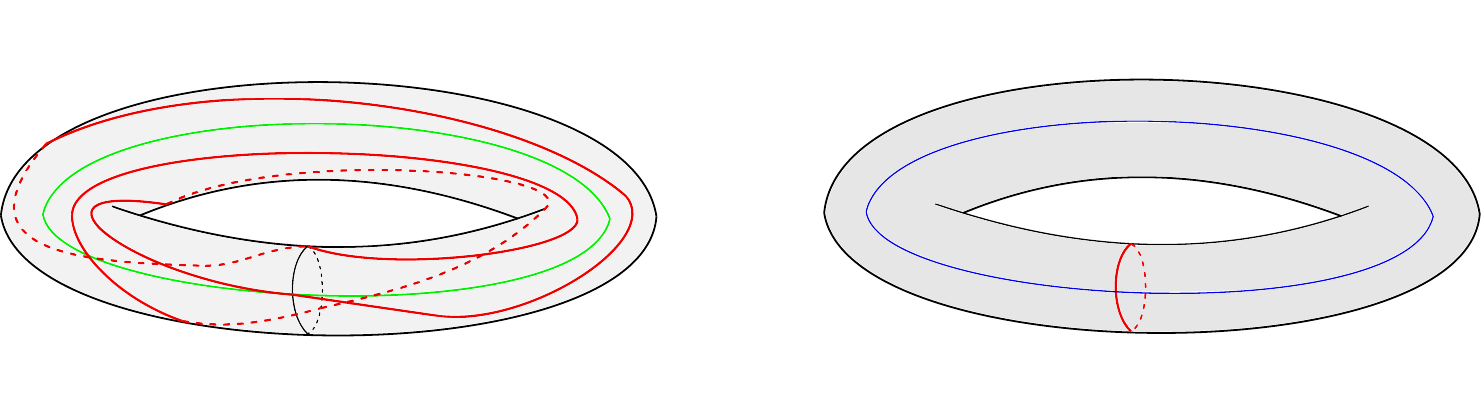}}%
    \put(0.68952188,0.00588857){\color[rgb]{0,0,0}\makebox(0,0)[lt]{\lineheight{1.25}\smash{\begin{tabular}[t]{l}$S^1 \times D^2$\end{tabular}}}}%
    \put(0.0715832,0.00565994){\color[rgb]{0,0,0}\makebox(0,0)[lt]{\lineheight{1.25}\smash{\begin{tabular}[t]{l}$(\partial S)_1 \times S^1$\end{tabular}}}}%
    \put(0,0){\includegraphics[width=\unitlength,page=2]{gluing0.pdf}}%
    \put(0.46177958,0.2567297){\color[rgb]{0,0,0}\makebox(0,0)[lt]{\lineheight{1.25}\smash{\begin{tabular}[t]{l}$\psi$\end{tabular}}}}%
    \put(0,0){\includegraphics[width=\unitlength,page=3]{gluing0.pdf}}%
  \end{picture}%
\endgroup%

	\caption{The meridian on the right is sent to the $(3,2)$-curve (in red) on the left. For a honest/integral open book, the meridian is sent to the longitude (in green).}
	\label{fig:gluing_map}
\end{figure}

\begin{figure}[htbp]
	\centering
	\def\svgwidth{\textwidth}
 \scalebox{0.5}{
\begingroup%
  \makeatletter%
  \providecommand\color[2][]{%
    \errmessage{(Inkscape) Color is used for the text in Inkscape, but the package 'color.sty' is not loaded}%
    \renewcommand\color[2][]{}%
  }%
  \providecommand\transparent[1]{%
    \errmessage{(Inkscape) Transparency is used (non-zero) for the text in Inkscape, but the package 'transparent.sty' is not loaded}%
    \renewcommand\transparent[1]{}%
  }%
  \providecommand\rotatebox[2]{#2}%
  \newcommand*\fsize{\dimexpr\f@size pt\relax}%
  \newcommand*\lineheight[1]{\fontsize{\fsize}{#1\fsize}\selectfont}%
  \ifx\svgwidth\undefined%
    \setlength{\unitlength}{312.1307855bp}%
    \ifx\svgscale\undefined%
      \relax%
    \else%
      \setlength{\unitlength}{\unitlength * \real{\svgscale}}%
    \fi%
  \else%
    \setlength{\unitlength}{\svgwidth}%
  \fi%
  \global\let\svgwidth\undefined%
  \global\let\svgscale\undefined%
  \makeatother%
  \begin{picture}(1,1.1752835)%
    \lineheight{1}%
    \setlength\tabcolsep{0pt}%
    \put(0,0){\includegraphics[width=\unitlength,page=1]{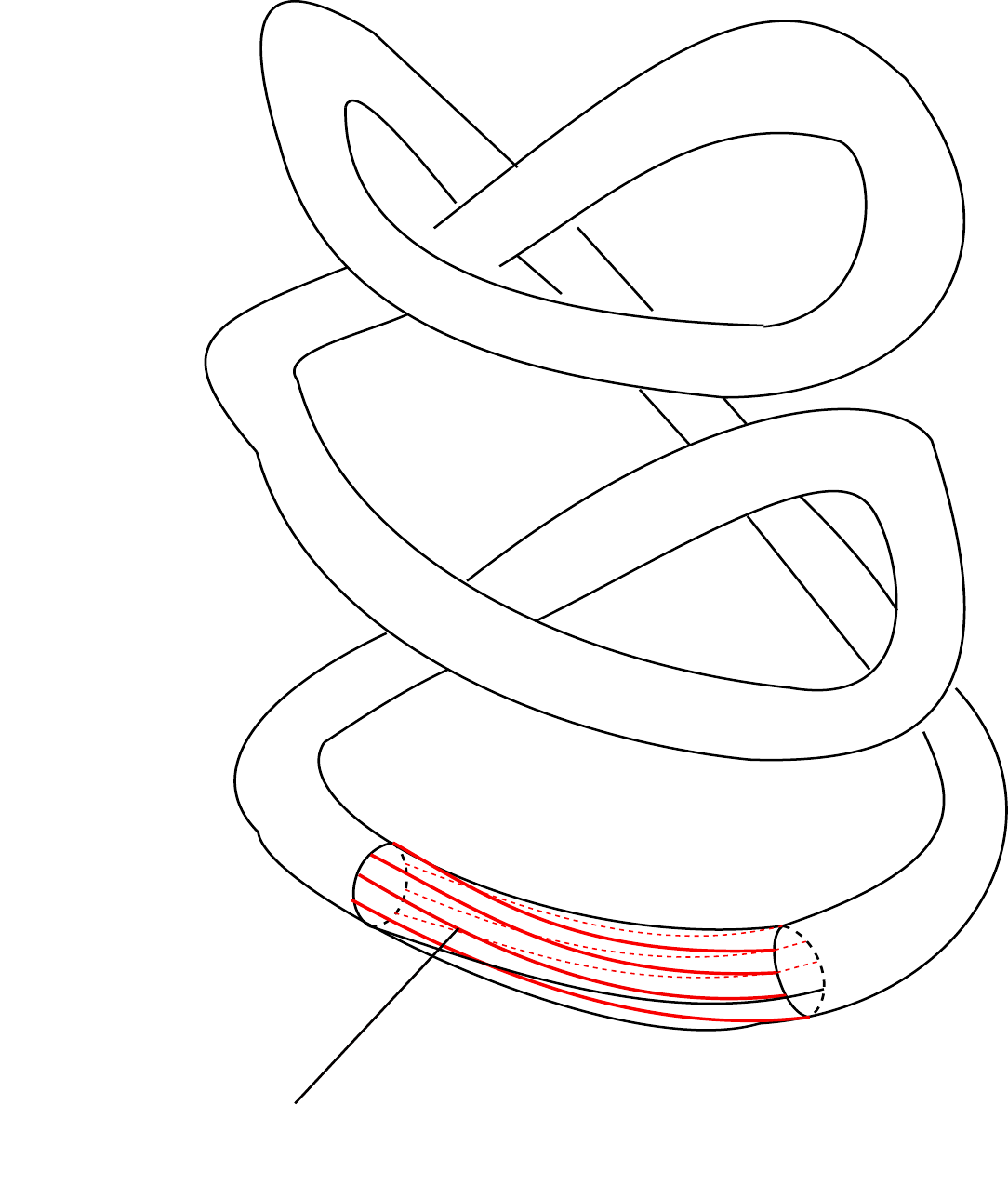}}%
    \put(-0.0037445,0.01104202){\color[rgb]{0,0,0}\makebox(0,0)[lt]{\lineheight{1.25}\smash{\begin{tabular}[t]{l}$(c_i^{-1},n_i)$-curve\end{tabular}}}}%
  \end{picture}%
\endgroup%
}
	\caption{Gluing of solid torus near permuting boundaries}
	\label{permglue}
\end{figure}

 Using this gluing map we define the total manifold of the corresponding rational book as $M_{\phi} = \MT(\Sigma, \phi) \cup_{\phi} (S^1 \times D^2)$. The pullback form is given by $\psi^*\alpha_K = (-rp-Kq)d\theta + (-rq + pK)d\phi$. We now extend this form using a form $f_0(r)d\theta + g_0(r)d\phi$. This form will be contact if and only if $f_0(r)g'_0(r) - f'_0(r)g_0(r) > 0$. Near $S^1 \times D^2$ (i.e. $r = 1$), $f_0(r) = -rp-qK$ and $g_0(r) = -rq + pK$. Near the core of $S^1 \times D^2$ (i.e. $r = 0$), $f_0(r) = 2-r^2$ and $g_0(r) = r^2$. We can then extend the functions $f_0$ and $g_0$ to define a contact form on whole of $S^1 \times D^2$. Moreover, this extension is unique. For more details on rational open books, we refer to \cite{BEM}.
	
\subsection{From integral open books to rational open books} \label{integral to rational}
One can see a rational open book of $M$ more explicitly as an abstract integral open book $\OB(\Sigma,\phi)$ with some modification in a neighborhood of the binding $\partial \Sigma$. Let $L$ be a connected component of $\partial \Sigma$. We take a solid torus neighborhood $U_L$ of $L$ and consider the identification of $\partial U_L$ with $\partial \MT(\Sigma,\phi)$. Let $\lambda$ and $\mu$ denote the reference longitude and meridian on $\MT(\Sigma,\phi)$. Then $\mu$ approaches $L$ along the $(1,0)$-curve on $\partial U_L$. We cut out $U_L$ from $\OB(\Sigma,\phi)$ and glue in a solid torus $U_0$, with longitude $\lambda_0$ and meridian $\mu_0$, by sending $\mu_0$ to $p\lambda + q\mu$. This is same as a topological $\frac{p}{q}$-surgery on $\OB(\Sigma,\phi)$ along $L$. Let $L_0$ denote the core of $U_0$. Note that $\mu$ now approaches $L_0$ via the $(p,q)$-curve on $\partial U_0$. The surgered manifold $M_{\frac{p}{q}}(L)$ admits an abstract open book decomposition with page $\Sigma$ and monodromy $\phi \circ \partial_{\frac{2\pi q}{p}}$. Here, $\partial_{\frac{q}{p}}$ denotes the $\frac{2\pi q}{p}$-rotation on the boundary component of $\Sigma$ aproaching $L_0$. We will denote this \emph{rational abstract open book} by $\ROB(\Sigma, \phi \circ \partial_{\frac{q}{p}})$. For $\vert p \vert = \pm 1$, this rational open book is an integral open book. 

\section{Associating contact structures to data sets}\label{marked data sets}
 Let $h$ be a pseudo-periodic homeomorphism on a closed surface $\Sigma$, such that $\Sigma$ can be decomposed into disjoint connected sub-surfaces $\Sigma_{g_i}$s so that $f_i = h\vert_{\Sigma_{g_i}}$ is an irreducible periodic map. In other words, $h$ can be decomposed into finitely many compatible irreducible Type $1$ actions.

\subsection{\textbf{Associating contact structures to Type $1$ data sets}} Let us first associate a contact structure to the data set corresponding to an irreducible Type $1$ action. The idea is to somehow associate a page and a monodromy of a rational open book. Since a $C_n$-action realized by a data set lives on a closed surface, to obtain a page with periodic homeomorphism one needs to remove disks around some of the non-trivial orbit points of that action. Therefore, we define something called a \emph{marked data set}. Let us first discuss the notion for a specific example.

\begin{exmp}
	Consider the data set $D_{\phi_0} = (5, 0; (1,5), (3,5), (1,5))$ that represents the homeomorphism $\phi_0$ of a genus-$2$ closed surface $\Sigma_2$, induced by a $\frac{2\pi}{5}$-rotation of a $10$-gon with opposite sides identified as in Figure \ref{rotatingon}. This is an irreducible Type $1$ periodic homeomorphism. This map has three fixed points, two of which have $\frac{2\pi}{5}$-rotation in their disk neighborhoods, while the third has a $\frac{4\pi}{5}$-rotation.
	
	Now consider the surface $\tilde{\Sigma}_2$, obtained from $\Sigma_2$ by removing two disk neighborhoods, both with rotation $\frac{2\pi}{5}$. Let $\tilde{\phi}_0$ denote the restriction homeomorphism $\phi_0|_{\tilde{\Sigma}_2}$. We fix an ordering of the cone points in the data set and let the last cone point denote the center of rotation in the covering polygon. We can index the ordering as $\{1,2,3\}$, and we mention the indexing numbers of the cone points around which we are removing the disks. In our case, if we fix the order of the cone points as written in $D$, then inclusion of $[1,3]$ in the data set will say that disks around the first and third cone points are removed. Another aspect we need to take into consideration is the direction of rotation. Note that the data set $D$ can be realized either by a  $+\frac{2 \pi}{5}$-rotation or by a  $-\frac{8\pi}{5}$-rotation. 
	
	\begin{figure}[htbp]
		\centering
	\def\svgwidth{\textwidth}
	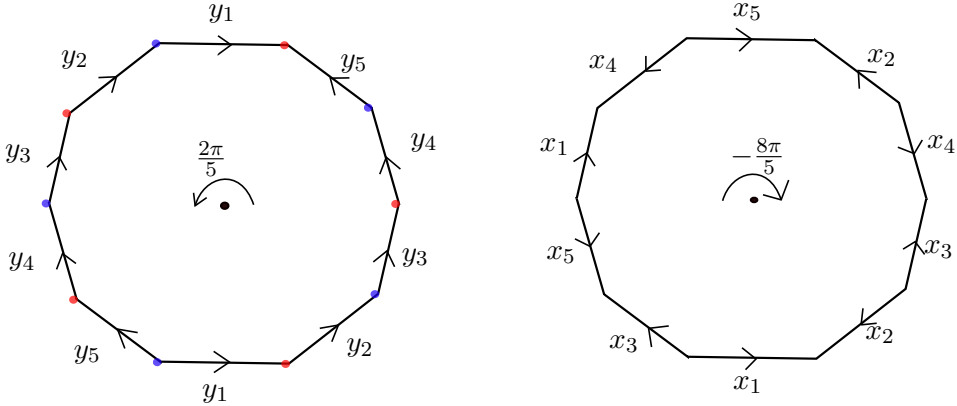
		\caption{Order $5$ rotation on a $10$-gon. The orbits among cone points are denoted by similar colours.}
		\label{rotatingon}
	\end{figure}

	\noindent The two cases will make a huge difference in terms of contact structures. So, we simply add a plus or minus sign with the first entry of the data set. In particular, a marked data set like $\hat{D}_1 = (5_+, 0; (1,5), (3,5), (1,5), [1,3])$ will represent the map $\tilde{\phi}_0$ on $\tilde{\Sigma_2}$. Similarly, $\hat{D}_2 = (5_-, 0; (1,5), (1,5), (3,5), [1,2,3])$ will represent a homeomorphism on a surface, obtained from $\Sigma_2$ by removing three disks around three fixed points, and the homeomorphism will be restriction of a homeomorphism on $\Sigma_2$ induced by a $-\frac{6\pi}{5}$-rotation of the $10$-gon.
\end{exmp} 

In general, let $D_{h_0} = (n_3, 0;(c_1,n_1),(c_2,n_2),(c_3,n_3))$ be a data set representing an irreducible periodic homeomorphism $h_0$ on a closed surface $\Sigma_{g_0}$. By Theorem \ref{res:1}, the conjugacy class of $h_0$ is realized by a $\frac{2 \pi c_3^{-1}}{n_3}$-rotation on an even sided polygon with appropriate side-pairing. This means there is $g \in \Homeo^+(\Sigma_{g_0})$ such that $ghg^{-1}$ is isotopic to $R_0$, where $R_0$ is the map induced on $\Sigma_0$ from the rotation on the polygon. Assume that $n_1, n_2 < n_3$. Thus, $h = R_0$ has a fixed point $x_{31}$ on $\Sigma_{g_0}$, around which $h$-induces a $\frac{2 \pi c_3^{-1}}{n_3}$-rotation. 
Moreover, there are orbits with $\frac{n}{n_i}$ points such that $h$ cyclically permutes the orbit points, and $h^{\frac{n}{n_i}}$ induces a $\frac{2 \pi c_i^{-1}}{n_i}$-rotation around each orbit point, for $i = 1,2$ (i.e. $h$ induces $\frac{2 \pi c_i^{-1}}{n_3}$-rotation). We have $N = (1 + \frac{n_3}{n_1} + \frac{n_3}{n_1})$ many cone points. Here also, the given data set can be realized either by a  $+\frac{2 \pi c_3^{-1}}{n_3}$-rotation or by a  $-\frac{2 \pi (n_3 - c_3^{-1})}{n_3}$-rotation.
 
\begin{definition}[Type $1$ marked data set]
A modified data set of the form $$\hat{D}_\pm = ({n_3}_\pm, 0; (c_1 , n_1),(c_2, n_2),(c_3,n_3), [j_1,\dots,j_k])$$ will be called a Type $1$ \emph{marked data set}.
\end{definition}

\noindent Thus, a Type $1$ marked data set $\hat{D}_\pm = ({n_3}_\pm, 0;(c_1,n_1),(c_2,n_2),(c_3,n_3),[j_1,$\\$\cdots,j_k])$ represents a surface $\hat{\Sigma}_{g_0}$, obtained from $\Sigma_{g_0}$ by removing some $l$-many open disks ($1 \leq l \leq N$), and the restriction homeomorphism ${\hat{R}_0}\vert_{\hat{\Sigma}_{g_0}}$. We associate the induced contact structure on $\ROB(\hat{\Sigma}_0, \hat{R}_0)$ to the marked data set $\hat{D}_\pm$. Note that the association of a contact structure to a marked data set as above is well-defined.	

\subsection{\textbf{Associating contact structures to Type $2$ data sets and their relation to Type $1$ data sets}} 

Similar to the previous section, one can define marked data sets of Type $2$. These are simply collection of compatible Type $1$ data sets with extra marking of points. We describe the association for some specific examples. 

\begin{exmp}\label{type 2 exmp 1}
	
	Consider the data set $(6,0 ; (1,2), (1,2), (1,3), (2,3))$ from Example \ref{(1,5)-compatible exmp}. This data set can be realized by combining two $(1,5)$-compatible Type $1$ data sets: 
	$$D_1 = (6,0;(1,2),(1,3),(1,6)) \text{ and } D_2 = (6,0;(1,2),(2,3),(5,6)).$$
	
	Now consider two marked data sets coming from these Type $1$ data sets: 
	\begin{gather*}
	\hat{D}_{\phi_1} = ((6_+, 0; (1,2), (1,3), (1,6), [1,3]) \text{ and } \\\hat{D}_{\phi_2} = ((6_-, 0; (1,2), (2,3),(5,6), [1,2,3]). 
	\end{gather*} Let $(\Sigma_{g_1}, \phi_1)$ and $(\Sigma_{g_2}, \phi_2)$ be the representative surfaces and monodromies corresponding to $\hat{D}_{\phi_1}$ and $\hat{D}_{\phi_2}$ respectively. We denote a compatible gluing by the following notation. A gluing represented by $(2:1) \sim (3:2)$ will mean gluing the second boundary component of the first surface to the third boundary component of the second surface. The notation : $([\hat{D}_{\phi_1}, \hat{D}_{\phi_2}], (3:1) \sim (3:2))$, will then represent the surface $\Sigma_{g_{12}}$, obtained by gluing an annulus between $\Sigma_{g_1}$ and $\Sigma_{g_2}$ along their boundary components with rotations $\frac{2\pi}{6}$ and $-\frac{2\pi}{6}$ respectively, and the homeomorphism $\phi_{12}$, obtained as a union of $\phi_1$ and $\phi_2$ and then extended by the identity map on the connecting annulus. In other words we have a $(2,3)$-compatible gluing of $\hat{D}_{\phi_1}$ and $\hat{D}_{\phi_2}$. 
	
	Now, $\phi_i$ can be isotoped, relative to boundary, to $h_i \circ \delta_i$, where $h_i \in \Mod(\Sigma_{g_i},\partial \Sigma_{g_i})$ and $\partial_i$ denotes the rotations on boundary components, for $i = 1,2$. More precisely, $\partial_1$ is the union of rotations on the four boundary components of $\Sigma_{g_1}$- one invariant boundary component with $\frac{2\pi}{6}$-rotation (corresponding to $(1,6)$ in $\hat{D}_{\phi_1}$) and three permuting boundaries with $\frac{2\pi}{6}$-rotation (corresponding to $(1,2)$ in $\hat{D}_{\phi_1}$) on each of them. Similarly, $\partial_2$ consists of three permuting boundaries with $-\frac{2\pi}{6}$-rotations, two permuting boundaries with $-\frac{4\pi}{6}$-rotations and one boundary with $-\frac{2\pi}{6}$-rotation. Let us rewrite $h_i$ as $\hat{h}_i \circ \beta_i$, for $i = 1,2$. Here, $\beta_i$ is the restriction of the free isotopy between $\phi_i$ and $h_i$ on the non-permuting boundary components (which are to be glued together) and $\hat{h}_i = \phi_i \cup \bar{\beta_i}$. It can be seen from Figure \ref{Fig 3} that the resulting homeomorphism $\phi_{12}$ on $\Sigma_{g_{12}}$ is given by $\hat{h}_1 \cup id. \cup \hat{h}_2$. 
	
\begin{figure}[H]
	\centering
	\def\svgwidth{\textwidth}
	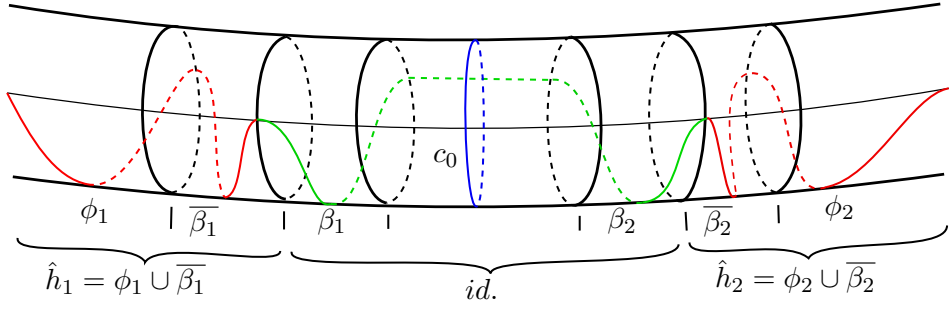
	\caption{Monodromy after gluing along compatible boundaries with rotations of different signs.}
	\label{Fig 3}
\end{figure}
\end{exmp}

\begin{exmp}\label{type 2 exemp 2}
	If we had the following Type $2$ collection of marked data sets : $\hat{D}_{\phi_1} = (6_+, 0; (1,2), (1,3), (1,6), [1,3])$ and $\hat{D}_{\phi_2} = (6_+, 0; (1,2), (2,3),(5,6), [1,2,3])$, then a similar analysis as in Example \ref{type 2 exmp 1} will show that $([\hat{D}_{\phi_1}, \hat{D}_{\phi_2}], (3:1) \sim (3:2))$ represents the surface $\Sigma_{g_{12}}$ with homeomorphism of the form $\hat{h}'_1 \cup T_{c_0} \cup \hat{h}_2$, as shown in Figure~\ref{Fig 4} below. 
\begin{figure}[htbp]
	\centering
\def\svgwidth{\textwidth}
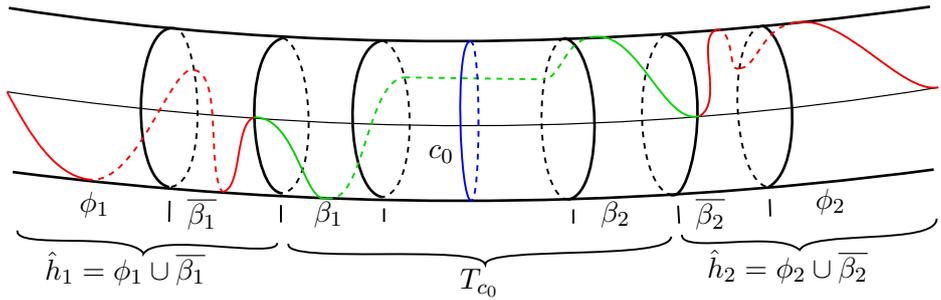
	\caption{Monodromy after gluing along compatible boundaries with rotations of same sign.}
	\label{Fig 4}
\end{figure}
	
\noindent Here, $T_{c_0}$ denotes positive Dehn twist along the curve $c_0$. Moreover, $\hat{D}_1 = (6_-, 0; (1,2), (1,3), (1,6),[1,3])$ and $\hat{D}_2 = (6_-, 0; (1,2), (2,3),(5,6), [1,2,3])$  will have a homeomorphism of the form $\hat{h}'_1 \cup T^{-1}_{c_0} \cup \hat{h}'_2$. One can similarly look at the resulting homeomorphisms for self compatible data sets.
\end{exmp}

Note that we order the cone points of a marked data sets that has been obtained by compatible gluing of two other marked data sets by writing the cone point entries of $\hat{D}_1$ followed by those of $\hat{D}_2$ and then removing the cone points which are killed in compatible or self-compatible gluing. Thus in general we can define a Type $2$ \emph{marked data set} as follows.

\begin{definition}
	A modified data set of the form $$\hat{D} = ({n_l}_\pm,g_0;(c_1,n_1),(c_2,n_2),\dots,(c_l,n_l), [j_1,\dots,j_k])$$ is called a Type $2$ \emph{marked data set}.
\end{definition}

\begin{proof}[Proof of Proposition \ref{association thm}]
	
We associate a contact structure to an arbitrary Type $2$ marked data set $\hat{D}$ in the following way. By \cite[Theorem 2.24]{PRS}, every Type $2$ data set can be constructed from finitely many compatible or self-compatible irreducible Type $1$ data sets. Thus, we first consider the finitely many rational open books associated to those compatible Type $1$ marked data sets and then inductively glue the pages and homeomorphisms as in Example \ref{type 2 exmp 1} and Example \ref{type 2 exemp 2}, to give the resultant rational open book. The contact structure compatible with this resultant rational open book gives the contact structure associated to $\hat{D}.$

\end{proof}

\section{Symplectic fillability of rational open books}\label{fillability}
\label{sec:symp_fill}
In this section we prove our main results. Before going into the proofs, we briefly review the notion of admissible transverse surgery on a contact manifold from \cite{BEM}.

\subsection{Admissible transverse surgery} Let $K \subset (M,\xi)$ be a transverse knot with a fixed framing $F$. 

\begin{definition}\label{transverse}
	A $\frac{q}{p}$-surgery on $K$ is called admissible if there exists a neighborhood $N \subset M$ of $K$ that is contactomorphic to a neighborhood $N_{r_0} = \{(r,\theta,z) \vert r \leq \sqrt{r_0}\}$ of the $z$-axis in $\mathbb{R}^3/(z \equiv z + 1)$ with the contact structure $\xi_0 = ker(dz + r^2 d\theta)$ such that $F$ goes to the product framing on $N_{r_0}$ and $- \infty < \frac{q}{p} < -\frac{1}{r_0}$.  
\end{definition}

If $M_K(\frac{q}{p})$ is obtained from $M$ by an admissible transverse surgery, then $M_K(\frac{q}{p})$ admits a natural contact structure $\xi_K(\frac{q}{p})$ on it. This contact structure is defined by a contact cut or reduction process. One takes a neighborhood of $K$ as in definition \ref{transverse} and considers the characteristic foliations on the tori at different radii from the central knot. Topologically, a $\frac{q}{p}$-surgery sends the $(p,q)$-curve on a torus to the boundary of a disk. For more on contact cuts we refer to \cite{BE}.  

In some cases, one can recover a contact rational open book via admissible transverse surgery on an honest contact open book. In particular, recall the rational contact open book $\ROB(\Sigma, \phi \circ R_{\frac{q}{p}})$ and the honest contact open book  $\OB(\Sigma, \phi)$ as described in section \ref{integral to rational} with $L = \partial \Sigma$. We assume that both $p$ and $q$ are positive.

\begin{lemma} \label{binding surgery}
  $\ROB(\Sigma, \phi \circ R_{\frac{q}{p}})$ is obtained from $\OB(\Sigma, \phi)$ via a	$-\frac{p}{q}$-transverse surgery on $L$.
\end{lemma} 

\begin{proof}
	Recall from section \ref{contact structure on rob} that we glued a solid torus $S^1 \times D^2$ to one of the boundary components of the mapping torus $\MT(\Sigma, \phi)$ by sending a meridian $\{\cdot\} \times \partial D^2$ to the $(p,q)$-curve (i.e. representing $p [l] + q [m]$ homology class) with slope $\frac{q}{p}$. In case of an honest open book, this attaching homeomorphism interchanges the meridian and longitude of the solid torus with that of the mapping torus. More precisely, $\{\cdot\} \times \partial D^2$ goes to $(1,0)$ and $S^1 \times \{\cdot\}$ goes to $(0,-1)$. Therefore, a $(p,q)$-curve on $\partial \MT(\Sigma, \phi)$ will be identified with a $(q,-p)$-curve on $\partial(S^1 \times D^2)$. 
	Now, there exists a neighborhood $N_L$ of the binding $L$ in $\OB(\Sigma, \phi)$ such that $N_L$ is contactomorphic to  $N_{r_0}$ as in Definition \ref{transverse}. So we can do a transverse $-\frac{p}{q}$ surgery along $L$. This amounts to attaching a disk along the $(q,-p)$-curve on $\partial N_{-\frac{p}{q}}$, which is same as attaching a solid torus to $\MT(\Sigma, \phi)$ along the $(p,q)$-curve on its boundary.       
\end{proof}	

\noindent Note that Lemma \ref{binding surgery} is also true for rational open books with multiple rotating boundary components. 
	
\begin{lemma}\label{irreducible filling}
Consider the contact structure $(M^3,\xi_0) = \ROB(\Sigma, R_0)$ associated to an irreducible marked data set $$ D = (n_+,0,; (c_1,n_1), (c_2,n_2), (c_3,n), [j_1,\cdots,j_k]),$$ where $[j_1,\cdots,j_k]$ is as mentioned in Theorem \ref{stein filling 1}. Then, $(M^3,\xi_0)$ is Stein fillable.
\end{lemma}

\noindent Note that by Equation \ref{eqn:riemann_hurwitz} in Definition \ref{defn:data_set}, $\Sigma$ has genus $\frac{n-1}{2}$. Therefore, $n$ has to be odd. Suppose the contact $(M^3,\xi)$ admits an open book decomposition $\OB(S, h)$ with $h$ freely isotopic to a periodic monodromy. Let $r_i$ be the $\textrm{FDTC}$ of the $i^{th}$ boundary component of $S$ and $\psi$ be the periodic representative of $h$.  The proof of Lemma \ref{irreducible filling} will require the following result~\cite[Theorems 4.1-4.2]{CH}.

\begin{theorem}[Colin--Honda, \cite{CH}]\label{colinho1}
If all the $r_i$ are positive, then $(M, \xi)$ is a uniquely Stein fillable $S^1$-invariant contact structure which is transverse to the $S^1$-fibers.
\end{theorem} 

\noindent The next result that we need is due to Baldwin-Etnyre~\cite[Theorem 3.2]{BE}. 

\begin{theorem}[Baldwin--Etnyre, \cite{BE}] \label{BE}
Let $K$ be a transverse knot in some contact manifold. Suppose $N$ is a standard neighborhood of $K$ such that the characteristic foliation on $\partial N$ is linear with slope $a$, where $n < a < n+1$ for some integer $n$. Then, for any rational number $s < n$, admissible transverse $s$-surgery on $K$ can also be achieved by Legendrian surgery on some Legendrian link in $N$.
\end{theorem}

\noindent We now prove Lemma \ref{irreducible filling}. 
\begin{proof}[Proof of Lemma \ref{irreducible filling}]
	
	Let us consider the marked data set $$D = (n_+,0,; (c_1,n), (c_2,n),(c_3,n), [1,2,3]).$$ We see that $\Sigma$ is a genus-$(\frac{n-1}{2})$ surface with three boundary components. Each of these three boundary components are invariant under $R_0$ and the $i$th boundary component rotates by an angle of $\frac{2\pi c_i^{-1}}{n}$ for $i =1,2,3$. In other words, $R_0$ is isotopic, relative to boundary, to a homeomorphism $h \circ \partial_1 \circ \partial_2 \circ \partial_3$, where $h$ is an element of $\Mod(\Sigma, \partial \Sigma)$ and $\partial_i$ represents the boundary rotation with $\textrm{FDTC} = \frac{c_i^{-1}}{n}$. It is clear that $h$ is freely isotopic to $h \circ \partial_1 \circ \partial_2 \circ \partial_3$. Since $R_0$ comes from a positive marked data set, $\frac{c_i^{-1}}{n} > 0$ for all $i$. Therefore, Theorem \ref{colinho1} implies that the contact open book $\OB(\Sigma,h)$ is uniquely Stein fillable. Let $(W^4, d\lambda)$ denote this Stein filling.   
	
\vspace{5mm}	
Now by Lemma \ref{integral to rational}, $\ROB(\Sigma, R_0) = \ROB(\Sigma, h \circ \partial_1 \circ \partial_2 \circ \partial_3)$ can be obtained from $\OB(\Sigma,h)$ by doing $-\frac{n}{c_i^{-1}}$-transverse surgery on the $i$th boundary component for $i = 1,2$ and $3$. Let $L_i$ be the $i$th binding component in $\OB(\Sigma,h)$ and let $N_i$ denote a standard contact neighborhood of $L_i$ such that the characteristic foliation on $\partial N_i$ has slope $a$ such that $-1 < a < 0$. Note that near the binding of an honest open book, the contact plane rotates so that the slope of the characteristic foliation goes from $0$ to $-\infty$. Therefore, one can always find such an $N_i$. Applying Theorem \ref{BE} for $n = -1$, we get that $\ROB(\Sigma, h \circ \partial_1 \circ \partial_2 \circ \partial_3)$ can be obtained from $\OB(\Sigma,h)$ by Legendrian surgery on some link in $N_1 \cup N_2 \cup N_3$. Each of these Legendrian surgeries along a knot in $(M,\xi)$ amounts to attaching a Stein $2$-handle to $(W^4,\lambda)$. Hence, $\ROB(\Sigma, R_0)$ is Stein fillable.
	
It is easy to see that the proof in the general case is exactly the same with fewer number of binding components to do surgery on. 
\end{proof}	

\noindent The final ingredient we need is the following result, which is a straightforward corollary of \cite[Theorem 1.3]{BEM}.

\begin{theorem}[Baldwin--Etnyre--van Horn-Morris, \cite{BEM}]\label{stein cobord}
	If $\OB(\Sigma,\phi_1)$ and $\OB(\Sigma,\phi_2)$ is Stein fillable, then $\OB(\Sigma, \phi_1 \circ \phi_2)$ is Stein fillable.
\end{theorem}

We are now ready to prove Theorem \ref{stein filling 1}.

\begin{proof}[Proof of Theorem \ref{stein filling 1}]
	
Here also we consider data sets with $3$-disk neighborhoods of cone points removed, i.e., with $[\hat{n},\hat{n},\hat{n}]$. The proof is similar for the remaining cases.

Let $\hat{D}_i = (n_+, 0; (c_{i1},n), (c_{i2},n), (c_{i3},n), [1,2,3])$ for $i = 1,2$ be two compatible irreducible marked data sets as in the hypothesis of Theorem \ref{stein filling 1}. Let $\ROB(\Sigma_{g_i}, R_i)$ be the rational open book associated to $\hat{D}_i$. Assume that $\hat{D}_1$ and $\hat{D}_2$ are $(2,3)$-compatible. Which means that $c_{12} + c_{23} \equiv 0 \pmod n$. According to our notation used in Example \ref{type 2 exmp 1}--\ref{type 2 exemp 2}, the resultant marked data set is represented by $([\hat{D}_1, \hat{D}_2], (2:1) \sim (3:2))$. 

As in Lemma \ref{irreducible filling}, we write the monodromy homeomorphism $R_i$ as a composition of an element of the relative mapping class group with fractional rotations near boundary. let $R_i = h_i \circ \partial_{i1} \circ \partial_{i2} \circ \partial_{i3}$ for $i = 1,2$. As described in Lemma \ref{irreducible filling}, both $\OB(\Sigma_{g_1},h_1)$ and $\OB(\Sigma_{g_2},h_2)$ are Stein fillable by Theorem \ref{colinho1}. 

Now recall condition $(2)$ in Theorem \ref{stein filling 1}: both $\partial (\Sigma_{g_1} \cup \Sigma_{g_2}) \cap \partial \Sigma_{g_1}$  and $\partial (\Sigma_{g_1} \cup \Sigma_{g_2}) \cap \partial \Sigma_{g_2}$ are non-empty. This implies that we can build $\Sigma_{g_1} \cup \Sigma_{g_2}$ by attaching $1$-handles to either $\Sigma_{g_1}$ or $\Sigma_{g_2}$. The reason behind this is the following. Let $C_b$ be a boundary component of $\Sigma_{g_1} \cup \Sigma_{g_2} \setminus \Sigma_{g_1}$. Let $C_0$ be the curve along which $\Sigma_{g_1}$ and $\Sigma_{g_2}$ are glued together, as shown in Figures~\ref{stein glue 1}.

\begin{figure}[h] 
	\centering
	\def\svgwidth{\textwidth}
	\scalebox{.85}{
\begingroup%
  \makeatletter%
  \providecommand\color[2][]{%
    \errmessage{(Inkscape) Color is used for the text in Inkscape, but the package 'color.sty' is not loaded}%
    \renewcommand\color[2][]{}%
  }%
  \providecommand\transparent[1]{%
    \errmessage{(Inkscape) Transparency is used (non-zero) for the text in Inkscape, but the package 'transparent.sty' is not loaded}%
    \renewcommand\transparent[1]{}%
  }%
  \providecommand\rotatebox[2]{#2}%
  \newcommand*\fsize{\dimexpr\f@size pt\relax}%
  \newcommand*\lineheight[1]{\fontsize{\fsize}{#1\fsize}\selectfont}%
  \ifx\svgwidth\undefined%
    \setlength{\unitlength}{435.94574364bp}%
    \ifx\svgscale\undefined%
      \relax%
    \else%
      \setlength{\unitlength}{\unitlength * \real{\svgscale}}%
    \fi%
  \else%
    \setlength{\unitlength}{\svgwidth}%
  \fi%
  \global\let\svgwidth\undefined%
  \global\let\svgscale\undefined%
  \makeatother%
  \begin{picture}(1,0.38107936)%
    \lineheight{1}%
    \setlength\tabcolsep{0pt}%
    \put(0,0){\includegraphics[width=\unitlength,page=1]{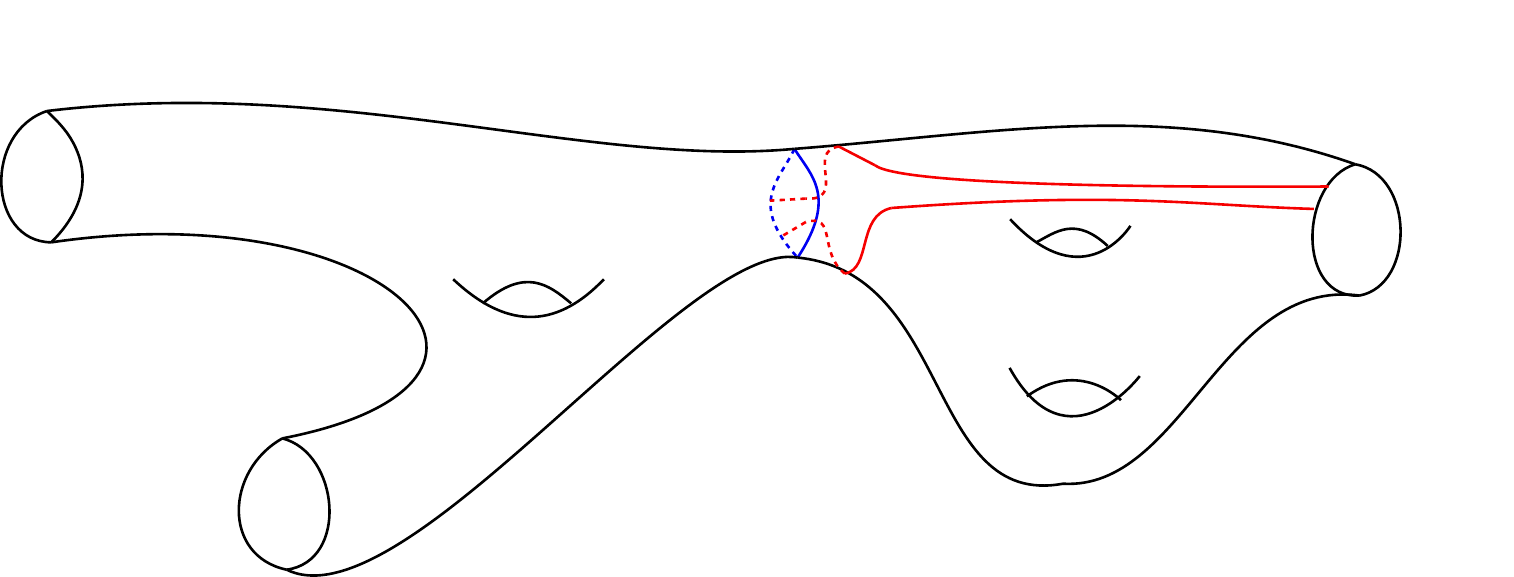}}%
    \put(0.94074613,0.22515207){\color[rgb]{0,0,0}\makebox(0,0)[lt]{\lineheight{1.25}\smash{\begin{tabular}[t]{l}$C_b$\end{tabular}}}}%
    \put(0,0){\includegraphics[width=\unitlength,page=2]{stg1.pdf}}%
    \put(0.47472443,0.36654736){\color[rgb]{0,0,0}\makebox(0,0)[lt]{\lineheight{1.25}\smash{\begin{tabular}[t]{l}$C_0$\end{tabular}}}}%
    \put(0.26790804,0.24753611){\color[rgb]{0,0,0}\makebox(0,0)[lt]{\lineheight{1.25}\smash{\begin{tabular}[t]{l}$\Sigma_1$\end{tabular}}}}%
    \put(0.72749304,0.18277918){\color[rgb]{0,0,0}\makebox(0,0)[lt]{\lineheight{1.25}\smash{\begin{tabular}[t]{l}$\Sigma_2$\end{tabular}}}}%
  \end{picture}%
\endgroup%
}
	\caption{Retracting the boundary component $C_b$ to the red curve. After retraction, the surface becomes homeomorphic to the boundary connected sum of $\Sigma_{g_1}$ with a $1$-handlebody.}
	\label{stein glue 1}
\end{figure}

\noindent We deformation retract $C_b$ to bring it near $C_0$ so that the resulting manifold is homeomorphic to the boundary connected sum of $\Sigma_{g_1}$ and a $1$-handlebody $H_1$. One can then extend $h_1$ to a homeomorphism $\hat{h}_1$ on $\Sigma_{g_1} \cup \Sigma_{g_2}$ by identity on $H_1$. Similarly we get $\hat{h}_2$ on $\Sigma_{g_1} \cup \Sigma_{g_2}$.

\noindent Since $\tilde{\Sigma} = \Sigma_{g_1}\cup \Sigma_{g_2}$ is obtained by attaching $1$ handles to either $\Sigma_{g_1}$ or $\Sigma_{g_2}$, $\OB(\tilde{\Sigma}, \tilde{h}_i)$ is Stein fillable for $i = 1,2$. Therefore, by Theorem \ref{stein cobord}, $\OB(\tilde{\Sigma}, \tilde{h}_1 \circ \tilde{h}_2)$ is Stein fillable.

 We now recall from Example \ref{type 2 exemp 2} the description of the resulting homeomorphism after compatible gluing of two marked data sets. According to that description the resulting homeomorphism on $\tilde{\Sigma}$ is given by $h_1 \cup T_{c_0} \cup h_2 \cup \bar{\partial}_1 \cup \bar{\partial}_2$, which is the same as the homeomorphism $\tilde{h}_1 \circ T_{C_0} \circ \tilde{h}_2 \circ \bar{\partial}_1 \circ \bar{\partial}_2$. Here, $\bar{\partial}_i$ denotes the union of rotations on the boundary components of $\Sigma_i$, except the ones that are used in the compatible gluing. In particular, the contact manifold associated to $([\hat{D}_1, \hat{D}_2], (c_{12},n)_1 \sim (c_{12},n)_2)$ is given by $\ROB(\tilde{\Sigma}, \tilde{h}_1 \circ T_{C_0} \circ \tilde{h}_2 \circ \bar{\partial}_1 \circ \bar{\partial}_2)$, as shown in Figure \ref{deco} below.
 
\begin{figure}[htbp] 
 	\centering
\def\svgwidth{\textwidth}
\begingroup%
  \makeatletter%
  \providecommand\color[2][]{%
    \errmessage{(Inkscape) Color is used for the text in Inkscape, but the package 'color.sty' is not loaded}%
    \renewcommand\color[2][]{}%
  }%
  \providecommand\transparent[1]{%
    \errmessage{(Inkscape) Transparency is used (non-zero) for the text in Inkscape, but the package 'transparent.sty' is not loaded}%
    \renewcommand\transparent[1]{}%
  }%
  \providecommand\rotatebox[2]{#2}%
  \newcommand*\fsize{\dimexpr\f@size pt\relax}%
  \newcommand*\lineheight[1]{\fontsize{\fsize}{#1\fsize}\selectfont}%
  \ifx\svgwidth\undefined%
    \setlength{\unitlength}{510.41573398bp}%
    \ifx\svgscale\undefined%
      \relax%
    \else%
      \setlength{\unitlength}{\unitlength * \real{\svgscale}}%
    \fi%
  \else%
    \setlength{\unitlength}{\svgwidth}%
  \fi%
  \global\let\svgwidth\undefined%
  \global\let\svgscale\undefined%
  \makeatother%
  \begin{picture}(1,0.26812873)%
    \lineheight{1}%
    \setlength\tabcolsep{0pt}%
    \put(0,0){\includegraphics[width=\unitlength,page=1]{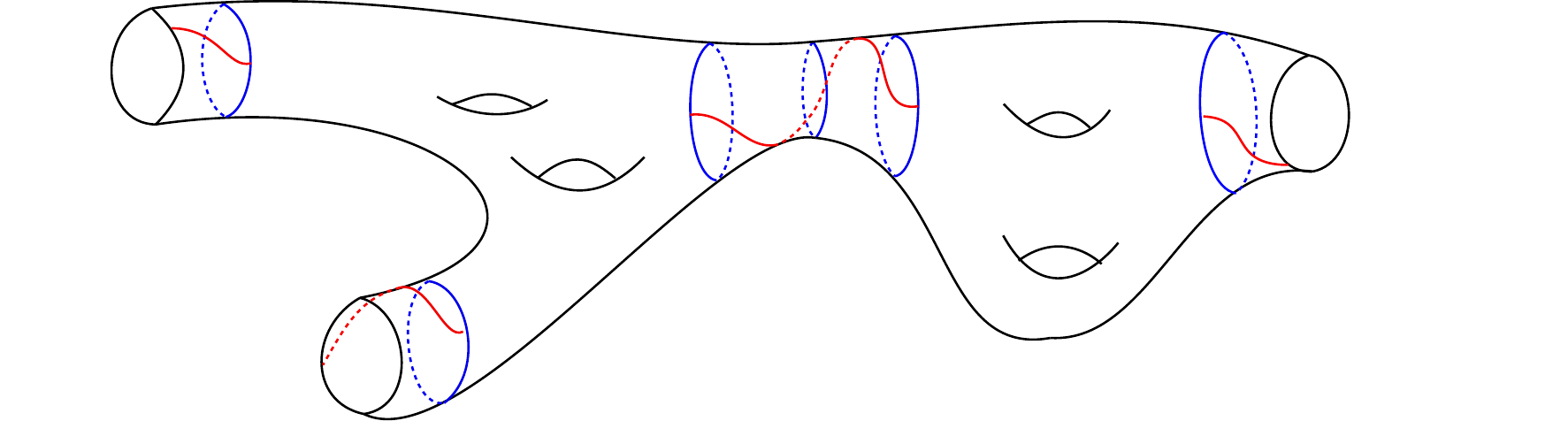}}%
    \put(-0.00065065,0.22420236){\color[rgb]{0,0,0}\makebox(0,0)[lt]{\lineheight{1.25}\smash{\begin{tabular}[t]{l}$\frac{2\pi c_{11}^{-1}}{n}$\end{tabular}}}}%
    \put(0.87811879,0.1906716){\color[rgb]{0,0,0}\makebox(0,0)[lt]{\lineheight{1.25}\smash{\begin{tabular}[t]{l}$\frac{2\pi c_{21}^{-1}}{n}$\end{tabular}}}}%
    \put(0.13486532,0.02050898){\color[rgb]{0,0,0}\makebox(0,0)[lt]{\lineheight{1.25}\smash{\begin{tabular}[t]{l}$\frac{2\pi c_{13}^{-1}}{n}$\end{tabular}}}}%
    \put(0.34867886,0.17751548){\color[rgb]{0,0,0}\makebox(0,0)[lt]{\lineheight{1.25}\smash{\begin{tabular}[t]{l}$h_1$\end{tabular}}}}%
    \put(0.66042833,0.15429976){\color[rgb]{0,0,0}\makebox(0,0)[lt]{\lineheight{1.25}\smash{\begin{tabular}[t]{l}$h_2$\end{tabular}}}}%
    \put(0.49064897,0.11639005){\color[rgb]{0,0,0}\makebox(0,0)[lt]{\lineheight{1.25}\smash{\begin{tabular}[t]{l}$T_{C_0}$\end{tabular}}}}%
    \put(0,0){\includegraphics[width=\unitlength,page=2]{deco.pdf}}%
  \end{picture}%
\endgroup%

 	\caption{Description of the resulting monodromy after a $(1,2)$-compatible gluing.}
 	\label{deco}
\end{figure} 

Note that $\OB(\tilde{\Sigma}, \tilde{h}_1 \circ T_{C_0} \circ \tilde{h}_2)$ can be obtained from  $\OB(\tilde{\Sigma}, \tilde{h}_1 \circ \tilde{h}_2)$ by attaching a Stein $2$-handle along $C_0$ in $\tilde{\Sigma}$. Therefore,  $\OB(\tilde{\Sigma}, \tilde{h}_1 \circ T_{C_0} \circ \tilde{h}_2)$ is Stein fillable. We can now use Lemma \ref{binding surgery} and Lemma \ref{irreducible filling} to conclude that $\ROB(\tilde{\Sigma}, \tilde{h}_1 \circ T_{C_0} \circ \tilde{h}_2 \circ \bar{\partial}_1 \circ \bar{\partial}_2)$ is Stein fillable.

\end{proof}	

As an application of Theorem \ref{stein filling 1},
we discuss two examples of marked data sets.

\begin{exmp} \label{stein fillable data set 1}

Consider the following marked data set.

 $$\hat{D} = (5_+,0;(3,5),(3,5),(1,5),(2,5), [1,2,3]).$$ 
 
\noindent Note that $\hat{D}$ can be realized by two $(1,2)$-compatible irreducible Type $1$ marked data sets : $\hat{D}_1 = (5_+,0,(3,5),(1,5),(3,5), [1,2,3])$ and $\hat{D}_2 = (5_+,0,$\\$(1,5),(2,5),(1,5), [1,3]).$ This example is described in Figure \ref{deco} for the values : $n = 5, c^{-1}_{11} = 2, c_{13}^{-1} = 2$ and $c_{21}^{-1} = 1$. The resulting surface is of genus $4$ with $2$ genera coming from each of its irreducible components. Theorem \ref{stein filling 1} then says that the contact structure associated to $\hat{D}$ is Stein fillable.
\end{exmp}

\begin{exmp} \label{stein fillable data set 2}
Consider the following marked data set. 
$$\hat{D} = (6_+,0;(1,2),(1,3),(1,3),(5,6),[4]).$$ 
\noindent This $\hat{D}$ can be realized by two $(3,3)$-compatible irreducible marked data sets: $\hat{D}_1 = (6_+,0,(1,2),(1,3),(1,6),[3])$ and $\hat{D}_2 = (6_+,0,(1,3),(5,6),(5,6),[2,3])$. Therefore, by Theorem \ref{stein filling 1}, the contact structure associated to $\hat{D}$ is Stein fillable. 	
\end{exmp}

\subsection{The case of self compatible gluing in marked data sets} The total surface after self compatible gluing within a marked data set can be seen in two steps. First one attaches a $1$-handle between the two compatible boundary components. Then one attaches a disk along the resulting connected boundary component. This is known as \emph{capping off} a boundary component. In the proof of Theorem \ref{stein filling 1}, we saw how compatible gluing between distinct surfaces preserves Stein fillability. The approach taken there breaks down for self-compatible gluing because the $1$-handles are attached within the same connected component. Moreover, the change in contact structure due to capping off a boundary in the page of a contact open book is not very clear. Note that Baldwin and Etnyre \cite{BE} have proved an interesting result that shows that capping off operation on certain universally tight contact structures may lead to overtwisted contact structures.   

\subsection{Exapmle of data set giving overtwisted contact structure} Consider the marked data set $D_0 = (6_-, 0; (1,2), (2,3), (5,6), [3])$. It represents an order $6$ homeomorphism $\phi_0$ on a genus $1$ surface $\Sigma_0$ with one boundary component. There is a $\frac{-2\pi}{6}$-rotation on the boundary. In terms of Lemma \ref{binding surgery}, this amounts to a $+6$ transverse surgery along the binding. In section $5$ of \cite{BEM}, the notion of \emph{integral resolution} was defined to go from a supporting rational open book of a contact structure to a supporting integral open book. In our case, the page $\Sigma_1$ of this supporting integral open book is obtained by gluing a $6$-punctured disk to $\Sigma_0$ along its boundary and the monodromy $\phi_1$ is given by composing the extension of $\phi_0$ to $\Sigma_1$ with a negative Dehn twist about a curve parallel to each boundary component in the punctured disk, as described in Figure \ref{Fig resol}. It is the easy to see that the monodromy of the integral open book is \textit{left veering}. Therefore, the supported contact structure on $OB(\Sigma_1,\phi_1)$ is overtwisted. 

\begin{figure}[htbp]
	\centering
	\def\svgwidth{\textwidth}
	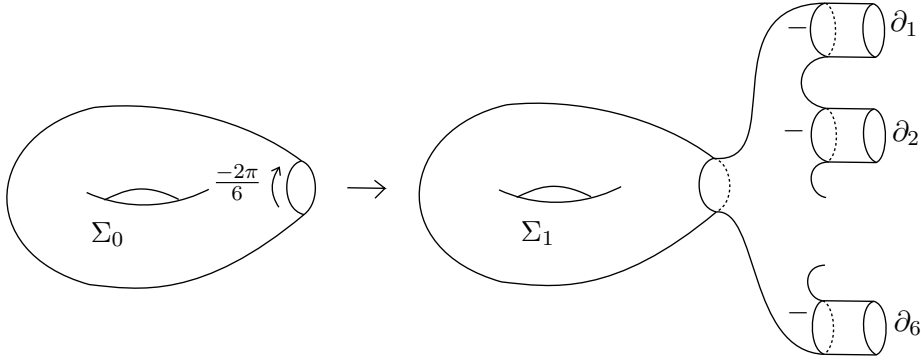
	\caption{Integral resolution of $ROB(\Sigma_0,\phi_0)$ to $OB(\Sigma_1,\phi_1)$.}
	\label{Fig resol}
\end{figure}

\section{Explicit symplectic fillings of rational open books}

\subsection{Mori's construction of symplectic filling} Mori \cite{Mo} constructed explicit strong symplectic filling of open books whose monodromy is composition of positive Dehn twists along disjoint curves. The part of his construction that we are interested in is the filling near binding of an open book. In particular, we will first look at the filling of an open book with identity monodromy. Note that $\partial(\Sigma \times D^2) = \OB(\Sigma,id)$. Let us consider the symplectic form $d\alpha \oplus 2s ds \wedge
d\phi$ on $\Sigma \times D^2$, where $d\alpha$ is an exact symplectic form on $\Sigma$ and $(s,\phi)$ are radial co-ordinates on $D^2$. We then attach a region, diffeomorphic to $S^1 \times D^3$, to $\Sigma \times D^2$ and extends the symplectic structure to all of $S^1 \times D^3$, so that the resulting manifold has a boundary contactomorphic to $\OB(\Sigma,id)$. Below we describe the procedure in more detail.

We can consider $\Sigma \times S^1$ sitting inside $(\Sigma \times \partial D^2 \times (0,1], (x, \phi, s)) \subset \Sigma \times D^2$. This induces a symplectic structure $\omega_0 = d(\alpha_K + s^2 d\phi)$ on $\Sigma \times \partial D^2 \times (0,1]$. First, we embed $\{\theta\} \times D^2 \subset S^1 \times D^2$ into $\mathbb{R}^3$ by the map $(r,\phi) \mapsto (h_1(r) cos\phi, h_1(r) sin\phi, h_0(r))$. Let $w = x + i y$ and $z = h_0(r)$. Here $h_0, h_1$ are smooth increasing functions defined on $[0,1]$ such that near $r = 1$, $h_0(r) = r -\frac{1}{2}$ and $h_1(r) = 1$, and near $r = 0$, $h_0(r) = \frac{r^2}{4}$ and $h_1(r) = r$. Thus, any point on the region $R_0 = \{(w,z) | h_0 \circ h^{-1}_1(|w|) \leq z \leq \frac{1}{2}\}$ can be represented by $z = h_0(r)$, $s h_1(r)$ and $arg(w) = \phi$, where $s \in [0,1]$ is determined by each point but $(0,0)$. By description, $R_0$ is diffeomorphic to $int(D^3)$. See Figure \ref{R0} below for a description of this embedding.

\begin{figure}[htbp]
	\centering
   	\def\svgwidth{\textwidth}
 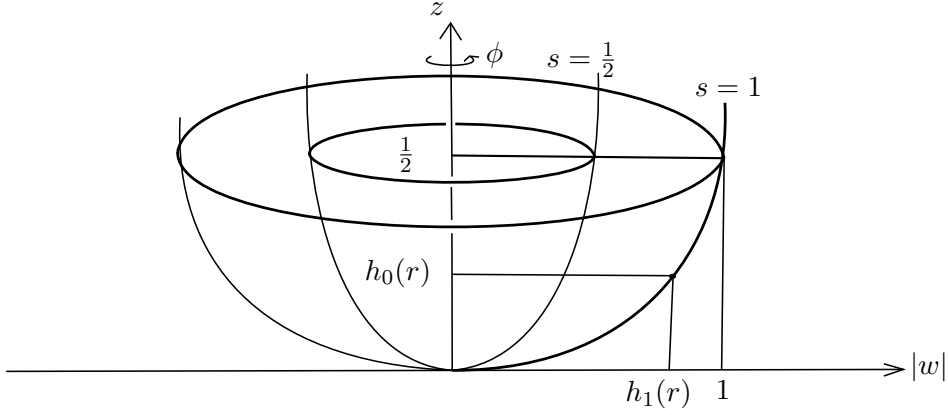
	\caption{Embedding of $\{\theta\} \times D^2$ in $R_0$}
	\label{R0}
\end{figure}

Using the gluing map between the boundary of a tubular neighborhood of binding and the boundary of the mapping torus, we can pull back the symplectic structure $\omega_0 = d(\alpha_K + s^2 d\phi)$ on $\Sigma \times \partial D^2 \times (0,1]$ to $S^1 \times N(\partial D^2) \times (0,1]$. One can then extend this symplectic structure by using the form $\omega = d(f(r)d\theta + s^2 g(r)d\phi)$, where $f$ and $g$ are real valued functions that interpolates between the contact structures on the mapping torus and on the solid torus neighborhood of the binding.

\subsection{Modification of Mori's construction for rational open books} Following Mori's construction in \cite{Mo}, our plan is to pullback the symplectic form $\omega_0$ 
over $S^1 \times N(\partial D^2) \times [1-\epsilon,1] \subset S^1 \times D^3$ via some gluing map between $S^1 \times N(\partial D^2) \times (0,1]$ and $\partial \Sigma \times (1-\epsilon,1] \times S^1 \times (0,1] \subset F \times D^2$ and then extend it to all of $S^1 \times R_0$ to produce a symplectic filling of $\ROB(\Sigma,\partial_{\frac{q}{p}})$.

We define the gluing map $\Psi$ as follows.

\begin{alignat*}{2}
\Psi: S^1 \times N(\partial D^2) \times (0,1] &\longrightarrow&  \partial \Sigma \times (1-\epsilon, 1] \times  S^1 \times (0,1] \\
(\theta,r,\phi,s) &\longmapsto& \   \ (p\theta + q\phi, -r, -q\theta + p \phi, s) 
\end{alignat*}

\noindent Then $\omega = \Psi^*\omega = d((-rp-Kq-qs^2)d\theta + (-r q + p K + p s^2) d\phi$ and $\omega \wedge \omega = 2s (p^2 + q^2) d\theta \wedge dr \wedge ds \wedge d\phi > 0$ for all $s \in (0,1]$. We want to extend this symplectic form to a $2$-form $\Omega = d(f(r,s) d\theta + g(r,s) d\phi)$, where $f(r,s) = f_0(r) - qs^2$ and $g(r,s) = g_0(r) + ps^2$, such that $\Omega \wedge \Omega > 0$ on $S^1 \times D^3$. 

A simple computation shows that $\Omega \wedge \Omega = -2s [p f'_0(r) + q g'_0(r)] d\theta \wedge dr \wedge ds \wedge d\phi)$. So, together with the contact condition, we want $f_0(r), g_0(r), p$ and $q$ satisfying the following conditions.

\begin{enumerate}
	\item Contact condition: $f_0(r)g'_0(r) - f'_0(r)g_0(r) > 0$ for all $r \in [0,1]$.
	\item Symplectic condition: $p f'_0(r) + q g'_0(r) < 0$ for all $r \in [0,1]$.
	\item  Near $r = 0$, $f_0(r) = 2H-r^2$ and $g_0(r) = r^2$. Here $H$ is a positive number. 
\end{enumerate}

\noindent Note that $f_0(1) = -p - q K, g_0(1) = -q + p K$ and $f'_0(r) = - 2r , g'_0(r) = 2r $ near $r = 0$. Thus, the symplectic condition implies that $-2r (p-q) < 0 \Longrightarrow p > q$.

We will investigate the above condition by cases.

\begin{enumerate}
	\item [Case 1] {$(p > 0, q < 0)$}: In this case we want to connect $(f_0(1),g_0(1)) = (-p - qK, -q + p K)$ and $(2H,0)$ so that conditions $(2)$ and $(3)$ are satisfied. By choosing $K$ large enough, we can always make sure that $(f_0(1),g_0(1))$ lies in the first quadrant. Moreover, we choose $H > f_0(1)$. Given this, we can always connect $(f_0(1),g_0(1))$ and $(2H,0)$ via $(f_0(r), g_0(r)$ such that $f'_0(r) < 0$ and $g'_0(r) > 0$ for all $r \in [0,1]$ such that conditions $(1)$ and $(3)$ are satisfied. However, note that such $f_0(r)$ and $g_0(r)$ may not satisfy condition $(2)$. For example, if $p = a$ and $q = -b$, where $a,b > 0$, then condition $(2)$ implies that $\frac{dg_0}{df_0}(r) > \frac{a}{b} > 0$ for all $r \in [0,1]$, which is not possible, as shown in Figure \ref{extension1}.
	
	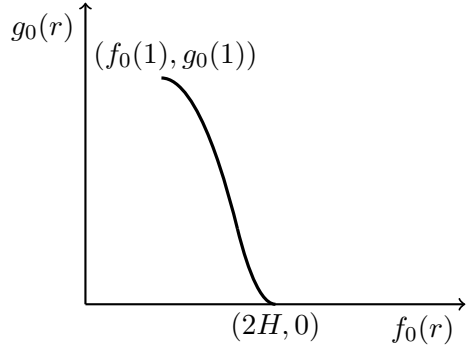
\begin{figure}[H]
		\begin{tikzpicture}
		\draw[thick,->] (0,0) -- (5,0) node[anchor=north east] {$f_0(r)$} ;
		\draw[thick,->] (0,0) -- (0,4)node[anchor=north east] {$g_0(r)$}; 
		\draw[very thick] (1,3) parabola  (2,1);
		\draw[very thick](2.5,0) parabola (2,1);
		\draw (2.5cm,1pt) node[anchor=north] {$(2H,0)$};
		\draw (2.4,3.3) node[anchor=east] {$(f_0(1),g_0(1))$};
		\end{tikzpicture}
		\caption{Extending the symplectic structure for Case $1$.}	
		\label{extension1}
	\end{figure} 
	
	\item[Case $2$]{$(p > 0, q > 0)$}: Taking $K$ large enough we can assume $f_0(1) < 0$ and $g_0(1) > 0$. Moreover, $(f'_0(1), g'_0(1)) = (- p, -q)$. Note that now $p f'_0(r) + q g'_0(r) < 0$ for all $r \in [0,1]$. Taking $H$ large enough we can connect $(f_0(1),g_0(1))$ and $(2H,0)$ satisfying conditions $(1)$ to $(3)$, as shown in Figure \ref{extension2}.
	
	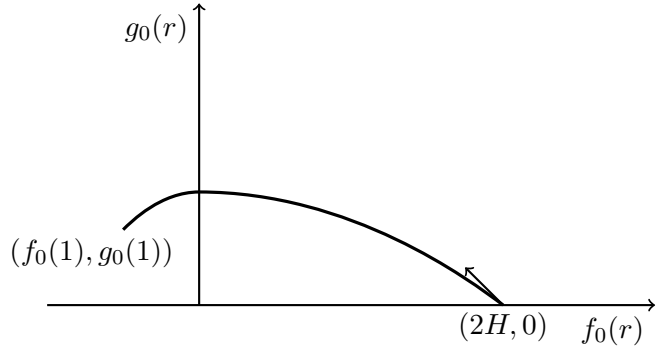
\begin{figure}[H]
		\begin{tikzpicture}
		
		\draw[thick,->] (-2,0) -- (6,0) node[anchor=north east] {$f_0(r)$} ;
		\draw[thick,->] (0,0) -- (0,4)node[anchor=north east] {$g_0(r)$}; 
		\draw[thick,->] (4,0) -- (3.5,0.5);
		\draw[very thick] (0,1.5) parabola  (4,0);
		\draw[very thick] (0,1.5) parabola (-1,1);
		\draw (4cm,1pt) node[anchor=north] {$(2H,0)$};
		\draw (-0.2,0.7) node[anchor=east] {$(f_0(1),g_0(1))$};
		\end{tikzpicture}
		\caption{Extending the symplectic structure for Case $2$.}
		\label{extension2}	
	\end{figure} 

	\item [Case 3] {$q < p < 0$}: For $K$ large enough, $f_0(1) > 0, g_0(1) < 0$ $f'_0(1) > 0$,$g'_0(1) > 0$. Then, any curve joining $(2H,0)$ and $(f_0(r), g_0(r))$ that satisfies conditions $(1)$ and $(3)$, must have a point $(f_0(r_0),g_0(r_0))$ such that both $f'_0(r_0)$ and $g'_0(r_0)$ are negative (see Figure \ref{extension3}). Thus, $p f'_0(r_0) + q g'_0(r_0) > 0$ and condition $(2)$ is violated. So, here we can not extend the symplectic form to the filling.

	\begin{figure}[H]
		\begin{tikzpicture}
		\draw[thick,->] (-3,0) -- (5,0) node[anchor=north east] {$f_0(r)$} ;
		\draw[thick,->] (0,-3) -- (0,4)node[anchor=north east] {$g_0(r)$}; 
		\draw (2.5cm,1pt) node[anchor=north] {$(2H,0)$};
		\draw[very thick, ->] (2.5,0) arc (0:320:2.5);
		\draw (3,-1.4) node[anchor=east] {$(f_0(1),g_0(1))$};
		\end{tikzpicture}
		\caption{Diagram for Case $3$.}	
		\label{extension3}
	\end{figure}
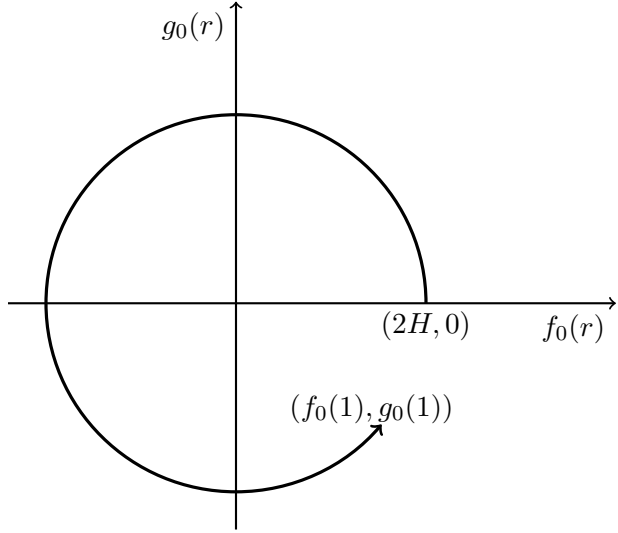 
	
\end{enumerate}
\noindent Thus, we have proved the following.

\begin{lemma} \label{rational filling}
	
	$\ROB(\Sigma,\partial_{\frac{q}{p}})$ is symplectically fillable for $p > q > 0.$
	
\end{lemma}

\noindent The above construction also takes care of the main step in the proof of Theorem \ref{symp filling 2}.

\begin{proof}[Proof of Theorem \ref{symp filling 2}]
Say, $h = T^{k_1}_{c_1} \circ T^{k_2}_{c_2} \circ \dots \circ T^{k_l}_{c_l}$. Let $m = \Sigma_{i=1}^l c_i$. Identify the $S^1$-interval of the mapping torus of $\ROB(\Sigma, \partial_{\frac{q}{p}})$ with $[0,1]/{0 \sim 1}$. Divide $[0,1]$ into sub-intervals $I_j = [\frac{j}{m},\frac{j+1}{m}]$ for $j \in {0,1,...,m-1}$. By Lemma \ref{rational filling}, we can construct a symplectic filling of $\ROB(\Sigma, \partial_{\frac{q}{p}})$. Let $W_0^4$ denote the filling. We attach $m$ Weinstein $2$-handles to $W_0^4$ along the appropriate $c_i$s, one each in the interval $I_j$ of the mapping torus of $\ROB(\Sigma,\partial_{\frac{q}{p}})$, to obtain a symplectic filling $W^4$ of $\ROB(\Sigma, h \circ \partial_{\frac{q}{p}})$.  
\end{proof}	

\begin{proof}[Proof of Theorem \ref{Stein filling 2}]
We will assume that the boundary $ \partial \Sigma$ is connected for simplicity of the argument. For multiple boundary components, our argument works near each boundary component. The boundary rotation, by hypothesis, is given by $\partial_{\frac{p}{q}} $.

We consider the honest open book $\left(\Sigma, h \right) $. By Giroux's result, for $h \in \text{Dehn}^+ (\Sigma, \partial \Sigma)$ the supported contact struture $ \xi$ admits a Stein-filling say $X$. The binding $\partial \Sigma$ has a neighborhood that is contactomorphic to the standard neighborhood $N_R$ in $\mathbb{R}^3/(z \equiv z+1) $ with the contact form $ dz + r^2 d\varphi$. By Lemma \ref{binding surgery}, the contact structure $\xi $ supported by  $ROB \left(\Sigma, h \circ \partial_{\frac{p}{q}}\right)$ is obtained by performing a $-\frac{p}{q}$ admissible transverse surgery on the binding of $OB \left( \Sigma, h \right)$. Since $-\frac{p}{q} < -1$, by Theorem \ref{BE}, we can realize this admissible transverse surgery by a sequence of Legendrian surgeries along some Legendrian link in the neighborhood of the binding. Thus, we can add Stein 2-handles to $X$ corresponding to each Legendrian surgery in the above sequence to get a Stein filling of $\xi$.      
\end{proof}

\bibliographystyle{plain} 
\bibliography{periodic_contact}

\begin{thebibliography}{10}

\bibitem{BEM}
Kenneth~L. Baker, John~B. Etnyre, and Jeremy Van Horn-Morris.
\newblock Cabling, contact structures and mapping class monoids.
\newblock {\em J. Differential Geom.}, 90(1):1--80, 2012.

\bibitem{BE}
John~A. Baldwin and John~B. Etnyre.
\newblock Admissible transverse surgery does not preserve tightness.
\newblock {\em Math. Ann.}, 357(2):441--468, 2013.

\bibitem{CH}
Vincent Colin and Ko~Honda.
\newblock Reeb vector fields and open book decompositions.
\newblock {\em J. Eur. Math. Soc.}, 15:443--507, 2013.

\bibitem{G1}
Jane Gilman.
\newblock Structures of elliptic irreducible subgroups of the modular group.
\newblock {\em Proc. London Math. Soc. (3)}, 47(1):27--42, 1983.

\bibitem{Gi}
Emmanuel Giroux.
\newblock G\'{e}om\'{e}trie de contact: de la dimension trois vers les
  dimensions sup\'{e}rieures.
\newblock In {\em Proceedings of the {I}nternational {C}ongress of
  {M}athematicians, {V}ol. {II} ({B}eijing, 2002)}, pages 405--414. Higher Ed.
  Press, Beijing, 2002.

\bibitem{HKM}
Ko~Honda, William~H. Kazez, and Gordana Mati\'{c}.
\newblock Right-veering diffeomorphisms of compact surfaces with boundary.
  {II}.
\newblock {\em Geom. Topol.}, 12(4):2057--2094, 2008.

\bibitem{Ke}
Steven~P. Kerckhoff.
\newblock The {N}ielsen realization problem.
\newblock {\em Ann. of Math. (2)}, 117(2):235--265, 1983.

\bibitem{Loi_Pier}
Andrea Loi and Riccardo Piergallini.
\newblock Compact stein surfaces with boundary as branched covers of b4.
\newblock {\em Inventiones mathematicae}, 143(2):325--348, 2001.

\bibitem{Mo}
Atsuhide Mori.
\newblock A note on {T}hurston-{W}inkelnkemper's construction of contact forms
  on {$3$}-manifolds.
\newblock {\em Osaka J. Math.}, 39(1):1--11, 2002.

\bibitem{PRS}
Shiv Parsad, Kashyap Rajeevsarathy, and Bidyut Sanki.
\newblock Geometric realizations of cyclic actions on surfaces.
\newblock {\em J. Topol. Anal.}, 11(4):929--964, 2019.

\bibitem{RV}
Kashyap Rajeevsarathy and Prahlad Vaidyanathan.
\newblock Roots of {D}ehn twists about multicurves.
\newblock {\em Glasg. Math. J.}, 60(3):555--583, 2018.

\bibitem{TW}
W.~P. Thurston and H.~E. Winkelnkemper.
\newblock On the existence of contact forms.
\newblock {\em Proc. Amer. Math. Soc.}, 52:345--347, 1975.

\bibitem{Th}
William~P. Thurston.
\newblock On the geometry and dynamics of diffeomorphisms of surfaces.
\newblock {\em Bull. Amer. Math. Soc. (N.S.)}, 19(2):417--431, 1988.

\bibitem{W}
H.~E. Winkelnkemper.
\newblock Manifolds as open books.
\newblock {\em Bull. Amer. Math. Soc.}, 79:45--51, 1973.

\end{thebibliography}

\end{document}